\documentclass[11pt]{amsart}
\usepackage{calc,amssymb,amsmath,ulem}
\usepackage{alltt}
\RequirePackage[dvipsnames,usenames]{color}

\input{kmacros3.sty}
\input{xy}
\input{chancery.sty}
\xyoption{all}
\usepackage{hyperref}
\usepackage{graphicx}
\usepackage[all,cmtip]{xy}
\usepackage{verbatim,xspace}
\usepackage[left=1.25in,top=1.09in,right=1.25in,bottom=1.09in]{geometry}
\usepackage{mabliautoref}
\usepackage{mathrsfs}

\renewcommand{\O}{\cO}
\newcommand{\QQ}{\bQ}

\newcommand{\Tt}{\mathbb{T}}

\def\>{\geq}
\def\<{\leq}

\def\mcO{\mathcal{O}}
\newcommand{\Fdiff}{F\text{-}\diff}

\def\div{\operatorname{div}}
\def\mod{\operatorname{mod}}
\def\lct{\operatorname{lct}}
\def\fpt{\operatorname{fpt}}
\def\mDiff{\operatorname{\mathbf{Diff}}}

\begin{document}
\numberwithin{equation}{theorem}
\title{The $F$-different and a canonical bundle formula}
\author{Omprokash Das}
\address{School of Mathematics, Tata Institute of Fundamental Research, 1 Homi Bhabha Road, Colaba, Mumbai 400005}
\email{omdas@math.tifr.res.in}
\author[Schwede]{Karl Schwede}
\address{Department of Mathematics, University of Utah, 155 S 1400 E Room 233, Salt Lake City, UT, 84112}
\email{schwede@math.utah.edu}

\thanks{The first named author was supported by NSF FRG Grant DMS \#1265285, NSF Grant DMS \#1300750 and Simons Grant Award \#256202.  }
\thanks{The second named author was supported in part by the
  NSF FRG Grant DMS \#1265261/1501115, NSF CAREER Grant DMS \#1252860/1501102 and a Sloan
  Fellowship.}

  \subjclass[2010]{14F18, 13A35, 14B05, 14D99, 14E99}

\begin{abstract}
We study the structure of Frobenius splittings (and generalizations thereof) induced on compatible subvarieties $W \subseteq X$.  In particular, if the compatible splitting comes from a compatible splitting of a divisor on some birational model $E \subseteq X' \to X$ (ie, $W$ is a log canonical center), then we show that the divisor corresponding to the splitting on $W$ is bounded below by the divisorial part of the different as studied by Ambro, Kawamata, Koll\'ar, Shokurov, and others.  We also show that difference between the divisor associated to the splitting and the divisorial part of the different is largely governed by the (non-)Frobenius splitting of fibers of $E \to W$.  In doing this analysis, we recover an $F$-canonical bundle formula by reinterpretting techniques common in the theory of Frobenius splittings.
\end{abstract}

\maketitle

\section{Introduction}

This paper explores several questions that show up naturally in the study of Frobenius splittings, and also appear naturally as one compares the theory of $F$-singularities to the theory of the singularities of the minimal model program.

\begin{question}[Frobenius splitting perspective]
Suppose $X$ is a Frobenius split variety with an $F$-splitting $\phi$ and that $W$ is compatibly $\phi$-split with induced splitting $\phi_W$.  Now, $\phi_W$ corresponds to a divisor $D$ on (the normalization of) $W$.
\begin{itemize}
\item{} What divisor is it?
\item{} Where does it come from?
\item{} How does it vary as the characteristic varies?
\end{itemize}
\end{question}

This question turns out to be much more subtle when $W$ has codimension at least 2 in $X$.  The case where $W$ has codimension $1$ has already been largely answered in \cite{DasStronglyFRegularInversion}.

As mentioned, a very closely related question also appears in the comparison of $F$-singularities with the singularities of the minimal model program.  Let us sketch the situation under some simplifying assumptions.

Suppose that $(X, \Delta)$ is a proper log canonical pair and that $W \subseteq X$ is a normal (or minimal for simplicity) log canonical center.  Suppose the codimension of $W$ in $X$ is at least $2$.  Then one would like there to be divisors  $M_W$ and $\Diff_W$ (also denoted $\Delta_{W,\Div}$) on $W$ such that
\begin{equation}
(K_X + \Delta)|_W \sim_{\bQ} K_W + \Diff_W + M_W
\end{equation}
where
\begin{itemize}
\item[(a)]  $\Diff_W$ is effective and $(W, \Diff_W)$ is log canonical (or in the minimal case, even KLT), and
\item[(b)]  $M_W$ is semi-ample.
\end{itemize}
While such objects do not exist in general, some weak versions of them do exist, at least on birational models of $W$, see \cite{KawamataSubadjunctionOne}, \cite{KawamataSubadjunction2}, \cite[Theorem 0.2]{AmbroShokurovsBoundaryProperty}, \cite{HaconLCInversionAdjunction}, \cite[Conjecture 7.13]{ProkhorovShokurovSecondMainTheorem} and \cite{CortiFlipsFor3FoldsAnd4Folds}.
These results are consequences of \emph{canonical bundle formulae}.
We state two relevant results for comparison. The first is itself a version of a related canonical bundle formula (also known as an adjunction formula) and the second is frequently used consequence of such formula in characteristic $0$ (and which itself is also an adjunction formula).


\begin{theorem*}\cite[Theorem 0.2]{AmbroShokurovsBoundaryProperty}
	Let $f:X\to Z$ be a proper morphism between two normal varieties over $\bC$ with $f_*\mcO_X=\mcO_Z$. Let $(X, \Delta\>0)$ be a log variety such that $K_X+\Delta\sim_\bQ f^*L$ for some $\bQ$-Cartier $\bQ$-divisor $L$ on $X$. Suppose that $(X, \Delta)$ is KLT near the generic fiber of $f$. Then we have:
	\begin{enumerate}
		\item $K_X+\Delta\sim_\bQ f^*(K_Z+\Diff_Z+M_Z)$, for some $\bQ$-divisors $\Diff_Z$ and $M_Z$ on $Z$.
		\item The $b$-divisor $\mathbf{K}+\mathbf{\mDiff}$ is $b$-Cartier.
		\item The $b$-divisor $\mathbf{M}$ is $b$-nef.
		\end{enumerate}
	\end{theorem*}

\begin{theorem*}\cite{KawamataSubadjunctionOne, KawamataSubadjunction2, FG12}
	Let $X$ be a normal projective variety over $\bC$ and $\Delta\geq 0$ an effective $\bR$-divisor such that $(X, \Delta)$ is log canonical. Let $W$ be a minimal log canonical center of $(X, \Delta)$. Then there exists an effective $\bR$-divisor $\Delta_W\geq 0$ on $W$ such that
	\[(K_X+\Delta)|_W\sim_\bR K_W+\Delta_W \]
	and $(W, \Delta_W)$ has KLT singularities.
	\end{theorem*}	
In \cite{SchwedeFAdjunction}, a characteristic $p > 0$ variant of the second theorem above was shown.  Let us discuss this setup in more detail.
In characteristic $p > 0$, the $F$-singularities analog of log canonical singularities are \emph{$F$-pure singularities} and log canonical centers correspond to \emph{$F$-pure centers}  (which are a generalization of compatibly split subvarieties).  For instance, every log canonical center becomes an $F$-pure center at least in an ambient $F$-pure pair.  Relevant for us, in characteristic $p > 0$, if $(X, \Delta)$ is sharply $F$-pure and the index of $K_X +\Delta$ is not divisible by $p$, then if $W \subseteq X$ is a normal $F$-pure center, there is a \emph{canonically determined} effective divisor $\Delta_{W, \Fdiff}$ on $W$ such that
\begin{equation}
\label{eq.FadjunctionFormula}
(K_X + \Delta)|_W \sim_{\bQ} K_W + \Delta_{W, \Fdiff}
\end{equation}
where
\begin{itemize}
\item[(a')] $(W, \Delta_{W, \Fdiff})$ is $F$-pure and if $W$ is minimal, $(W, \Delta_{W, \Fdiff})$ is strongly $F$-regular (an analog of KLT singularities).
\end{itemize}
Note in this case there doesn't initially appear to be a moduli part, there is a single effective divisor.  Regardless, because of \autoref{eq.FadjunctionFormula}, it might also be natural to expect that there is a Frobenius-version of the canonical bundle formula lurking behind the scenes.  This is exactly what we will show.

Another way to phrase our initial question is:
\begin{question}[Singularities of the MMP perspective]
Suppose that $(X, \Delta)$ is a sharply $F$-pure pair and $W \subseteq X$ is a normal $F$-pure center.  Consider $\Delta_{W, \Fdiff}$ the $F$-different on $W$.
\begin{itemize}
\item{} How does $\Delta_{\Fdiff}$ compare with $\Diff_W$ and $M_W$?
\item{} Where does it come from?
\item{} How does it vary as the characteristic varies?
\end{itemize}
\end{question}

Our first result is:

\begin{theoremA*}\textnormal{[\autoref{cor.FDiffGeqDivDiff}]}
Suppose that $(X, \Delta)$ is a sharply $F$-pure (and hence log canonical) pair and $W \subseteq X$ is a log canonical center (and hence an $F$-pure center).  Let $\nu : W^{\textnormal{N}} \to W$ be the normalization of $W$ and let $\Diff_{W^{\textnormal{N}}}=\Delta_{W^N, \div}$ and $\Delta_{W^{\textnormal{N}}, \Fdiff}$ be the different and $F$-different respectively.  Then $\Diff_{W^{\textnormal{N}}}=\Delta_{W^N, \div} \leq \Delta_{W^{\textnormal{N}}, \Fdiff}$.
\end{theoremA*}

It then becomes very natural to study the difference $\Delta_{W^{\textnormal{N}}, \Fdiff} - \Diff_W$ which should be viewed as some characteristic $p > 0$ analog of the moduli part of the different.  Already we know it is an effective divisor.

In characteristic zero, $M_W$, the moduli part of the different comes from analyzing a family.  Consider the following situation.
Let $(X', \Delta') \xrightarrow{\pi} (X, \Delta)$ be a birational model of $X$ with $K_{X'} + \Delta' = \pi^*(K_X + \Delta)$, such that $\Delta' \geq 0$ and $E$ is a normal prime Weil divisor with discrepancy $-1$ such that $\pi(E) = W$.  Then $\pi_E : E \to W$ can be viewed as a family, indeed a flat family if $W$ is a curve.  Note $E$ has coefficient $1$ in $\Delta'$, and so since $K_{X'} + \Delta'$ is pulled back from the base, if $\Delta_E$ is the ordinary different of $K_{X'} + \Delta'$ on $E$ (in particular, $(K_{X'} + \Delta')|_E = K_E + \Delta_E$), then $K_E + \Delta_E \sim_{\bQ, \pi_E} 0$.  It is then natural to try to study the moduli part of the different via a canonical bundle formula, i.e., find a divisor $\Delta_W \geq 0$ such that $\pi_E^*(K_W + \Delta_W) \sim_{\bQ} K_E + \Delta_E$.

It turns out that using a simple method coming from the origins of Frobenius splitting theory, we obtain a canonical bundle formula in characteristic $p > 0$.

\begin{theoremB*}\textnormal{[\autoref{thm.CanonicalBundleFormla}, \autoref{cor.IsFrobSplitting}]}
Suppose that $\pi : E \to W$ is a proper map between normal $F$-finite integral schemes with $\pi_* \O_E = \O_W$. Suppose also that $\Delta \geq 0$ is a $\bQ$-divisor on $E$ such that $(p^e - 1)(K_E + \Delta)$ is linearly equivalent to the pullback of some Cartier divisor on $W$. Suppose that the generic fiber of $(E, \Delta)$ over $W$ is Frobenius split. Then there exists a canonically determined $\bQ$-divisor $\Delta_W \geq 0$ on $W$ such that $\pi^*(K_W + \Delta_W) \sim_{\bQ} K_X + \Delta$.

Furthermore, we can describe the support of $\Delta_W$ as follows.  Assume that $W$ is 1-dimensional (which we may certainly do if we just care about the support of $\Delta_W$) and that $\Delta_{\mathrm{vert}}$ is the vertical part of $\Delta$.  If additionally the fibers of $\pi$ are geometrically normal, then $\pi^*\Delta_W - \Delta_{\mathrm{vert}}$ is nonzero precisely over those points $t \in W$ where the fiber $(E_t, \Delta_t)$ is not Frobenius split.
\end{theoremB*}

Indeed, we also obtain versions of the above result which hold when $\Delta$ is not necessarily effective, see \autoref{cor.GenCanonicalBundleFormula}.  In some cases, we also obtain results on the singularities of $(W, \Delta_W)$, see \autoref{rem.EasyCanBundleSingularities}.
It also follows that the $\Delta_W$ we construct is essentially the one obtained in \autoref{eq.FadjunctionFormula} when $W$ is an $F$-pure center that is also a log canonical center and $E \to W$ a divisor with discrepancy $-1$ mapping to $W$, see \autoref{cor.FadjunctionViaCanBundle}.

\begin{remark}
For those coming from the Frobenius splitting perspective, this result can be specialized follows.  Suppose $X$ is Frobenius split and the splitting extends to a resolution of singularities $X' \to X$, and there is a compatibly split divisor $E \subseteq X'$ mapping to the necessarily compatibly split $W \subseteq X$.  Then the splitting on $W$, and the corresponding divisor on $W$, is largely governed by the fibers of $E \to W$ which are not Frobenius split (in a way somewhat compatible with the splitting of $X'$).
\end{remark}

In characteristic $p > 5$, if $X$ is 3-dimensional, $(X, \Delta)$ is sharply $F$-pure and $W$ is a 1-dimensional minimal log canonical center, then we can obtain $E \to W$ satisfying the first part of the above theorem via the MMP \cite{HaconXuThreeDimensionalMinimalModel,BirkarExistenceOfFlipsMinimalModels,BirkarWaldronMoriFiberSpaces}. In some cases, it is also possible to reduce to the case of integral fibers by employing base change, semi-stable reduction and results on $\overline{\mathcal{M}}_{0,n}$ (see \cite[Theorem 4.8]{DasHacon}). This gives us a precise way to describe the $F$-different, see \autoref{alg.GeometricFDifferentAlg} for details.

We also tackle the question of how the $F$-different behaves as the characteristic varies.  To do this, we first need to assume the weak ordinarity conjecture (which implies that log canonical singularities become $F$-pure after reduction to characteristic $p > 0$ \cite{MustataSrinivasOrdinary,TakagiAdjointIdealsAndACorrespondence}).

\begin{theoremC*}\textnormal{[\autoref{thm.ModuliPartOfFdifferentMoves}]}
Let $(X, \Delta \geq 0)$ be a normal pair in characteristic zero with $(p^e-1)(K_X + \Delta)$ Cartier for some $e>0$, and $W$ a normal LC-center of $(X, \Delta)$. Assume that the $\mathbf{b}$-divisor $\mathbf{K+\Delta_{\div}}$ descends to $W$;\footnote{Meaning that the ${\bf b}$-divisor $K_{W'} + \Delta_{W'}$ can be pulled back from $W$.} and in particular $K_W+\Delta_{W, \div}$ is $\bQ$-Cartier.
We consider the behavior of $(X_p, \Delta_p)$ after reduction to characteristic $p \gg 0$.

Assume the weak ordinarity conjecture.
Let $Q \in W$ be a point which is not the generic point of $W$. Then there exist infinitely many primes $p > 0$ such that if $\Delta_{W_p, \Fdiff} \geq (\Delta_{W,\Div})_p$ is the $F$-different of $(X_p, \Delta_p)$ along $W_p$, then $\Delta_{W_p, \Fdiff} = (\Delta_{W,\Div})_p$ near $Q$.  In other words, $Q$ is not contained in $\Supp(\Delta_{W_p, \Fdiff} - (\Delta_{W,\Div})_p)$.
\end{theoremC*}
\vskip 12pt
\noindent
{\it Thanks:}  The authors would like to thank Christopher Hacon, Ching-Jui (Ray) Lai, Zsolt Patakfalvi, David Speyer, Shunsuke Takagi and Chenyang Xu for many useful discussions.  We would also like to thank Christopher Hacon, Shunsuke Takagi and the referee for numerous useful comments on previous drafts.  We also thank David Speyer and Chenyang Xu for stimulating discussions with the second author where \autoref{ex.EllipticCurveExample} was first worked out.

\section{Preliminaries}

\begin{convention}
Throughout this article, all schemes are assumed to be separated and excellent.  All schemes in characteristic zero are of essentially finite type over a field.  All schemes in characteristic $p > 0$ are assumed to be $F$-finite.  Whenever we deal with a pair $(X, \Delta)$, it is assumed that $X$ is normal and that $K_X + \Delta$ is $\bQ$-Cartier, unless specified otherwise. If $X$ is an integral scheme, by $K(X)$ we mean either the field of rational functions of $X$ and also the constant sheaf of rational functions on $X$, depending on the context.  Whenever $\pi : X' \to X$ is a proper birational map of normal integral schemes, we always choose $K_{X'}$ and $K_{X}$ such that $\pi_* K_{X'} = K_X$ (in other words, we implicitly fix the canonical {\bf b}-divisor).
\end{convention}

\subsection{Maps and $\bQ$-divisors.}
\label{subsec.MapsAndQDivisors}
Suppose that $X$ is an $F$-finite normal integral scheme.  Given a nonzero map $\phi : F^e_* \O_X \to \O_X$, we have $\phi \in \Hom(F^e_* \O_X, \O_X) \cong \Gamma(X, F^e_* \O_X((1-p^e)K_X))$, and hence $\phi$ corresponds to an effective divisor linearly equivalent to $(1-p^e)K_X$.
\begin{itemize}
\item{} We use $D_{\phi}$ to denote this effective divisor corresponding to $\phi$.
\item{} We write $\Delta_{\phi} := {1 \over p^e - 1} D_{\phi}$ and note that $\Delta_{\phi} \sim_{\bQ} -K_X$.
\end{itemize}
It is easy to see that this process can be reversed.  Given an effective divisor $\Delta$ such that $\O_X( (1-p^e)(K_X + \Delta)) \cong \O_X$, one obtains a nonzero map $\phi = \phi_{\Delta}^e$ with $\Delta_{\phi} = \Delta$.  The choice of $\phi$ is also unique up to pre-multiplication by units of $\Gamma(X, \O_X)$.  One can also form the self-composition of $\phi =: \phi^e$.  For instance $\phi^{2e} = \phi^e \circ (F^e_* \phi^e) \in \Hom(F^{2e}_* \O_X, \O_X)$ and $\phi^{ne} = \phi^e \circ (F^e_* \phi^{(n-1)e}) = \phi^{2e} \circ (F^{2e}_* \phi^{(n-2)e}) = \ldots$.  It is not difficult to check that $\Delta_{\phi} = \Delta_{\phi^2} = \Delta_{\phi^3} = \ldots$.  For additional discussion of this correspondence between maps and divisors, see \cite[Section 4]{BlickleSchwedeSurveyPMinusE}.

Now suppose that $\sL$ is an arbitrary line bundle on the $F$-finite normal integral scheme $X$.  Then using the same argument, a map $\phi : F^e_* \sL \to \O_X$ yields an effective $\bQ$-divisor $\Delta_{\phi}$ with $\O_X( (1-p^e)(K_X + \Delta_{\phi}) \cong \sL$.  As before, self-(twisted)-composition of this map yields the same divisor as $\Delta_{\phi}$ (for details, see \cite[Lemma 4.0.1]{BlickleSchwedeSurveyPMinusE}).  Conversely, given an effective $\bQ$-divisor $\Delta \geq 0$ such that $(p^e - 1)(K_X + \Delta)$ is Cartier, then setting $\sL = \O_X( (1-p^e)(K_X + \Delta))$ we obtain a map $\phi_{\Delta} : F^e_* \sL \to \O_X$.

\subsubsection{Maps and non-effective $\bQ$-divisors}
\label{subSec.MapsAndNonEffective}

We now recall how to interpret non-effective $\Delta$.  The main idea is that we are still given a map from $F^e_* \sL$ but that the image need not be in $\O_X$, it might be in some fractional ideal in $K(X)$.  Indeed, suppose we are given an $\O_X$-linear map $\phi : F^e_* \sL \to K(X)$.  Then it is easy to see (at least locally) that for some effective Cartier divisor $E$ on $X$ that $\phi( F^e_* \sL( (1-p^e)E)) \subseteq \O_X$.  This gives us an effective divisor $\Delta'$ such that $\O_X((1-p^e)(K_X + \Delta')) \cong \sL((1-p^e)E)$.

\begin{definition}
\label{def.NoneffectiveDelta}
Set $\Delta_{\phi} = \Delta' - E$.
\end{definition}

Note that if $\phi : F^e_* \sL \to K(X)$ induces $\Delta_{\phi}$ then $\phi(F^e_* \sL) \subseteq \O_X(G)$ where $G$ is some effective Weil divisor supported where $\Delta_{\phi}$ is not effective.  Indeed, simply reflexify $\phi(F^e_* \sL)$.

\begin{lemma}
With notations as above, $\Delta_{\phi}$ is independent of the choice of $E$.
\end{lemma}
\begin{proof}
Suppose that $E_1, E_2$ are two effective Cartier divisors with $\phi( F^e_* \sL( (1-p^e)E_i)) \subseteq \O_X$.  Without loss of generality, we may assume that $E_1 \leq E_2$.  Then we have the following composition
\[
\mu : F^e_* \sL( (1-p^e)E_1 + (1-p^e)(E_2 - E_1)) \hookrightarrow F^e_* \sL( (1-p^e)E_1) \xrightarrow{\psi} \O_X.
\]
The map $\psi$ yields $\Delta_1 = \Delta_{\psi} - E_1$ and the composition yields $\Delta_2 = \Delta_{\mu} - E_2$.  On the other hand, straightforward local computation shows that $\Delta_1 + (E_2 - E_1) = \Delta_2$.  Hence $\Delta_1 - E_1 = \Delta_2 - E_2$ as claimed.
\end{proof}

With notation as above, suppose that we fix an embedding $\sL \subseteq K(X)$.  Then our map $\phi = \phi^e : F^e_* \sL \to K(X)$ yields a map $\phi^e : F^e_* K(X) \to K(X)$.  We can then form $\phi^{2e} = \phi \circ (F^e_* \phi) : F^{2e}_* K(X) \to K(X)$ and more generally $$\phi^{ne} = \phi \circ (F^e_* \phi^{(n-1)e}) : F^{ne}_* K(X) \to K(X).$$  We then restrict $\phi^{ne}$ to $F^{ne}_* \sL^{1 + p^e + \ldots + p^{(n-1)e}}$ yielding:
\begin{equation}
\label{eq.NoneffectiveMapComposition}
\phi^{ne} : F^{ne}_* \sL^{1 + p^e + \ldots + p^{(n-1)e}} \to K(X).
\end{equation}


\begin{lemma}
With notations as above, $\Delta_{\phi^{ne}} = \Delta_{\phi^e}$.
\end{lemma}
\begin{proof}
Choose $E$ such that $\phi^e( F^e_* \sL( (1-p^e)E)) \subseteq \O_X$.  It follows immediately that $\phi^{ne}( F^{ne}_* \sL^{1 + p^e + \ldots + p^{(n-1)e}}( (1-p^{ne})E)) \subseteq \O_X$ since that is how composition works for effective divisors.  Define $\Delta'^e$ and $\Delta'^{ne}$ corresponding to the restricted maps $\phi^e|_{F^e_* \sL( (1-p^e)E)}$ and $\phi^{ne}|_{F^{ne}_* \sL^{1 + \ldots + p^{(n-1)e}}( (1-p^{ne})E)}$ respectively.  We know $\Delta'^e - E =: \Delta_{\phi^e}$ and also that $\Delta'^{ne} - E =: \Delta_{\phi^{ne}}$.  But $\Delta'^{ne} = \Delta'^e$ since they correspond to compositions of the same map (see \cite[Lemma 4.1.2]{BlickleSchwedeSurveyPMinusE} and \cite[Theorem 3.11]{SchwedeFAdjunction}).  The conclusion follows.
\end{proof}

\subsection{$p^{-e}$-linear maps and base extension}
\label{subsec.pMapsAndBaseExtension}

The following subsection is only utilized briefly in \autoref{alg.GeometricFDifferentAlg} and can be skipped on a first reading.

We start this subsection by  summarizing some results from \cite{SchwedeTuckerTestIdealFiniteMaps}.
Suppose that $f : Z \to W$ is a finite dominant map of normal $F$-finite integral schemes and $\phi_W : F^e_* \sL_W \to K(W)$ is an $\O_W$-linear map from a line bundle $\sL_W$ on $W$.  Then $\phi_W$ corresponds to a divisor $\Delta_W$ as above.  Let $\Tt : f_* K(Z) \to K(W)$ be a nonzero map between the fraction fields of $Z$ and $W$ respectively.  Using an argument similar to that above, the map $\Tt|_{f_* \O_Z}$ gives us a Weil divisor $\mathcal{R}_{\Tt}$ which should be thought of as a type ramification divisor.  If $f$ is separable, then we typically assume that $\Tt$ is the field trace in which case $\mathcal{R}_{\Tt}$ is exactly the usual ramification divisor.

\begin{lemma}
With notation as above, fix a divisor $K_W$ so that $\sL_W = \O_W( (1-p^e)(K_W + \Delta_W)) \subseteq K(W)$.  This lets us extend $\phi_W$ to $\phi_W : F^e_* K(W) \to K(W)$.
\begin{itemize}
\item[(a)]  There is a unique map $\phi_Z : F^e_* K(Z) \to K(Z)$ so that $\Tt \circ f_* \phi_Z = \phi_W \circ \Tt$.  This map is called the $\Tt$-transpose of $\phi_W$.
\item[(b)]  $\mathcal{R}_{\Tt} \sim K_{Z/W} = K_Z - f^* K_W$ and hence we can fix $K_Z$ to be the divisor $f^* K_W + \mathcal{R}_{\Tt} $.
\item[(c)]  Via the inclusion $\sL_W \subseteq K(W)$ we obtain $\sL_Z := f^* \sL_W \subseteq K(Z)$ and thus obtain a map $\phi_Z := \phi_Z|_{\sL_W} : F^e_* \sL_W \to K(Z)$ with $\Delta_Z$ corresponding to $\phi_Z$.  With this notation, $f^* (K_W + \Delta_W) = K_Z + \Delta_Z$ and hence $\Delta_Z - f^* \Delta_W = -\mathcal{R}_{\Tt}$.
\end{itemize}
\end{lemma}
\begin{proof}
Part (a) is simply \cite[Proposition 5.4]{SchwedeTuckerTestIdealFiniteMaps}.

Part (b) follows from viewing the restriction $\Tt|_{f_* \O_Z}$ as a rational section of the sheaf \mbox{$\sHom(f_* \O_Z, \O_W) \cong f_* \O_Z(K_Z - f^* K_W)$}.
In the case that all divisors $\Delta_W$ and $\Delta_Z$ are effective, part (c) is simply \cite[Theorem 5.7]{SchwedeTuckerTestIdealFiniteMaps}.  But by the sort of twistings applied in \autoref{subSec.MapsAndNonEffective}, one easily reduces to the effective case and the result follows.
\end{proof}

With notation as above, suppose that $\Delta_W$ is effective, or in other words that we are given a map $\phi_W : F^e_* \sL_W \to \O_W$.  We immediately see that $\Delta_Z = f^* \Delta_W - \mathcal{R}_{\Tt}$ is not necessarily effective (even when $f$ is separable, $\Tt$ is the trace map and so $\mathcal{R}_{\Tt}$ is the ramification divisor).  Now consider
\[
\phi_Z : F^e_* \sL_Z = F^e_* \O_Z( (1-p^e)(K_Z + f^* \Delta_W - \mathcal{R}_{\Tt})) \to K(Z),
\]
and notice that if it so happens that $(1-p^e)\mathcal{R}_{\Tt}$ is Cartier, then we can always restrict $\phi_Z$ to $F^e_* \sL_Z ( (1-p^e) \mathcal{R}_{\Tt}) = F^e_* \O_Z((1-p^e)(K_Z + f^* \Delta_W ))$ and obtain a map corresponding to $f^* \Delta_W$,
\[
\phi_Z : F^e_* \O_Z( (1-p^e)(K_Z + f^* \Delta_W)) \to \O_Z.
\]
Observe that the image is contained in $\O_Z$ since $f^* \Delta_W$ is an effective divisor.  Now for simplicity suppose that $\sL_W = \O_W((1-p^e)(K_W + \Delta_W)) = \O_W$.  Then we obtain $F^e_* \O_Z( (1-p^e) \mathcal{R}_{\Tt}) \to \O_Z$.

We apply this in the following setting.

\begin{proposition}
\label{prop.keyBaseExtensionsPMapsForFamilies}
Suppose that $\pi : X \to C$ is a finite type flat family with geometrically integral generic fiber and $X$ normal over a regular integral $F$-finite $1$-dimensional scheme $C$ with $\pi_* \O_X = \O_C$.  Suppose that $g : D \to C$ is a finite map from a regular integral scheme, $Y$ is the normalization of the component of $X \times_C D$ dominating $D$  and $\pi_Y : Y \to D$ the induced flat family.  Choose a map $\Tt : g_* K(D) \to K(C)$.

\begin{itemize}
\item[(a)]  Then there is an induced map $\Tt_{Y/X} : K(Y) \to K(X)$ extending $\Tt$.  Furthermore if $Y = X \times_C D$ then $\mathcal{R}_{\Tt_{Y/X}} = \pi_Y^* \mathcal{R}_{\Tt}$.  Finally, if $K(C) \subseteq K(D)$ is separable and $\Tt = \Tr_{D/C}$, then $\Tt_{Y/X} = \Tr_{Y/X}$.
\end{itemize}

Further suppose that a nonzero $\phi_C : F^e_* \O_C \to \O_C$ extends to a map $\phi_X : F^e_* \O_X \to K(X)$, that $K(C) \subseteq K(D)$ is separable and that $\Tt = \Tr_{D/C}$.

\begin{itemize}
\item[(b)] Then the $\Tt$-transpose of $\phi_C$,  $\phi_D : F^e_* \O_D((1-p^e)\mathcal{R}_{\Tt}) \to \O_D$ extends to a map $\phi_Y : F^e_* \O_Y( (1-p^e) \mathcal{R}_{\Tt_{Y/X}}) \to K(Y)$.
\end{itemize}
\end{proposition}
\begin{proof}
We begin by analyzing this at the level of function fields.

We have $K(C) \subseteq K(X)$ a geometrically integral (and hence purely transcendental) finitely generated extension of fields.  Consider the map $\Tt : K(D) \to K(C)$.  We tensor with $K(X)$ and obtain:
\[
\xymatrix{
K(D) \otimes_{K(C)} K(X) \ar@{=}[d] \ar[r] & K(C) \otimes_{K(C)} K(X) \ar[d] \\
K(Y) \ar[r]_{\Tt_{Y/X}} & K(X)
}
\]
defining the map $\Tt_{Y/X}$.  It clearly extends $\Tt$ proving the first part of (a).  In the case that $K(C) \subseteq K(D)$ is separable and $\Tt = \Tr_{D/C}$, we notice that the basis $\{\ldots, b_i \ldots\}$ for $K(D)$ over $K(C)$ also yields a basis for $K(Y)$ over $K(X)$.  Then for any element $x \in K(X)$ and $d \in K(D)$, $xd$ acts on the basis first by $d$ and then by $x$.  It follows immediately that $\Tr_{Y/X}(xd) = x\Tr_{Y/X}(d) = x \Tt_{Y/X}(d) = \Tt_{Y/X}(xd)$ and so by linearity, $\Tt_{Y/X} = \Tr_{Y/X}$.  This proves the second part of (a).

We now turn to the statement that $\mathcal{R}_{\Tt_{Y/X}} = \pi_Y^* \mathcal{R}_{\Tt}$ from (a).  Viewing $\Tt$ as a rational section of $\sHom_{\O_C}(g_* \O_D, \O_C)$ we see that $\Tt_{Y/X}$ is obtained as the pullback of that same rational section to $\sHom_{\O_X}((g_Y)_* (\O_X \otimes_{\O_C} \O_D), \O_X)$.  Because we assumed that $Y = X \times_C D$, $\O_Y = \O_X \otimes_{\O_C} \O_D$ and we are done (please forgive the slight abuse of notation where we left off some necessary pushforwards and sheafy inverse images).  This finishes the proof of (a).

Now we handle the statement in (b) involving the map $\phi_C : F^e_* \O_C \to \O_C$ and its extension $\phi_X : F^e_* \O_X \to K(X)$.  We take the $\Tt$-transpose $\phi_D$ of $\phi_C$ and the $\Tt_{Y/X}$-transpose $\phi_Y$ of $\phi_X$.  We need only to verify that $\phi_Y$ extends $\phi_D$ at the level of field of fractions.  Since $\Tt = \Tr_{D/C}$ is the field trace, by \cite[Lemma 3.3, Proposition 4.1]{SchwedeTuckerTestIdealFiniteMaps} we see that $\phi_D = \phi_C \otimes_{K(C)} K(D)$ and $\phi_Y = \phi_X \otimes_{K(X)} K(Y) = \phi_X \otimes_{K(C)} K(D)$. It follows immediately that $\phi_Y$ extends $\phi_D$ since $\phi_X$ extends $\phi_C$.
\end{proof}

\begin{remark}
It would be natural to try to remove the hypothesis that $K(C)\subseteq K(D)$ is separable in (b) above.  We do not see how to do that however.
\end{remark}

\subsection{Singularities}

Suppose that $(X, \Delta)$ is a pair with $X$ an integral $F$-finite scheme, $\Delta$ a $\bQ$-divisor, and $K_X + \Delta$ $\bQ$-Cartier (we do not assume that $\Delta \geq 0$).

\begin{definition}[Log canonical singularities]
We say that $(X, \Delta)$ is \emph{sub-log canonical} if for any proper birational map $\pi : Y \to X$ from a normal scheme $Y$, the coefficients of $K_Y - \pi^*(K_X + \Delta)$ are all $\geq -1$.  We say that $(X, \Delta)$ is \emph{log canonical} if it is sub-log canonical and $\Delta \geq 0$.
\end{definition}

\begin{definition}[$F$-pure singularities]
\label{def.FPure}
Suppose further that $p$ does not divide the index of $K_X + \Delta$.  We say that $(X, \Delta)$ is \emph{sub-$F$-pure} if for all sufficiently divisible $e>0$ we have $\O_X \subseteq \Image\Big(\phi_{\Delta}^e : F^e_* \O_X(  (1-p^e)(K_X + \Delta)) \to K(X)\Big)$.  Here $\phi^e_{\Delta}$ is the map corresponding to $\Delta$ as in \autoref{subsec.MapsAndQDivisors}.
We say that $(X, \Delta)$ is \emph{$F$-pure} if it is sub-$F$-pure and $\Delta \geq 0$.  In this case $\phi_{\Delta}^e : F^e_* \O_X(  (1-p^e)(K_X + \Delta)) \twoheadrightarrow \O_X$ surjects.

More generally, if $(X, \Delta)$ is $F$-pure and $I \subseteq \O_X$ is an ideal sheaf and $t \geq 0$ is any real number, then we say that $(X, \Delta, I^t)$ is \emph{sharply $F$-pure} if
\[
\phi_{\Delta}^e : F^e_* I^{\lceil t(p^e-1)\rceil}\cdot \O_X(  (1-p^e)(K_X + \Delta)) \twoheadrightarrow \O_X
\]
surjects for any (equivalently all sufficiently divisible \cf the argument of \autoref{lem.OneSplitImpliesMore} below) $e > 0$.  In this setting, the \emph{$F$-pure threshold of $(X, \Delta, I)$}, denoted $\fpt(X, \Delta, I)$, is defined to be
\[
\sup\big\{ c \geq 0\;|\; (X, \Delta, I^c) \text{ is sharply $F$-pure} \big\}.
\]
\end{definition}

\begin{remark}
In \cite{TakagiAdjointIdealsAndACorrespondence}, Takagi gave a slightly different definition of $F$-pure pairs for non-effective divisors $\Delta$.  We believe that these two notions are indeed distinct, but we are unsure which is better in general.  The notion we work with in this paper is harder to satisfy, but we do not know if it corresponds to sub-log canonicity even assuming the weak ordinarity conjecture \cite{MustataSrinivasOrdinary}.
\end{remark}

\begin{lemma}
\label{lem.OneSplitImpliesMore}
Using the notations of \autoref{def.FPure}, suppose that $$\O_X \subseteq \Image\Big(\phi_{\Delta}^e : F^e_* \O_X(  (1-p^e)(K_X + \Delta)) \to K(X)\Big)$$ for some $e>0$.  Then for all integers $n \geq 1$,
\[
\O_X \subseteq \Image\Big(\phi_{\Delta}^{ne} : F^{ne}_* \O_X(  (1-p^{ne})(K_X + \Delta)) \to K(X)\Big).
\]
\end{lemma}
\begin{proof}
Since $\phi(F^e_* \O_X( (1-p^e)(K_X + \Delta))) \supseteq \O_X$, we see by twisting by appropriate line bundles, that $\phi(F^e_* \O_X( (1-p^{2e})(K_X + \Delta))) \supseteq \O_X( (1-p^e)(K_X + \Delta))$.  So
\[
\begin{array}{rl}
& \phi^2\big(F^{2e}_* \O_X( (1-p^{2e})(K_X + \Delta))\big) \\
= & \phi\big(F^e_* \phi(F^e_* \O_X( (1-p^{2e})(K_X + \Delta)))\big) \\
\supseteq & \phi\big(F^e_* \O_X( (1-p^e)(K_X + \Delta))\big) \\
\supseteq & \O_X.
\end{array}
\]
In general, if $\phi^n\big(F^{ne}_* \O_X( (1-p^{ne})(K_X + \Delta))\big) \supseteq \O_X$, then
\[
\begin{array}{rl}
& \phi^{n+1}\big(F^{(n+1)e}_* \O_X( (1-p^{(n+1)e})(K_X + \Delta))\big) \\
= & \phi\big(F^e_* \phi^n(F^{ne}_* \O_X( (1-p^{(n+1)e})(K_X + \Delta)))\big) \\
\supseteq & \phi\big(F^e_*\O_X( (1-p^e)(K_X + \Delta))\big) \\
\supseteq & \O_X.
\end{array}
\]
\end{proof}

Unlike the case where $\Delta$ is effective, it is not clear to us whether $(X, \Delta)$ being sub-$F$-pure implies that $\phi_{\Delta}^e$ has $\O_X$ in its image for \emph{all} $e$ such that $(p^e - 1)(K_X + \Delta)$ is $\bQ$-Cartier.  The problem is that $\phi^{2e}_{\Delta}$ doesn't factor through $\phi^e_{\Delta}$ when $\Delta$ is non-effective.  However, when $X$ is normal and $1$-dimensional, this is not an issue.

\begin{lemma}
Suppose that $X$ is normal and $1$-dimensional and $(X, \Delta)$ is sub-$F$-pure.  Then for every $e > 0$ such that $(p^e-1)(K_X + \Delta)$ is $\bQ$-Cartier, we have that $\O_X$ is in the image of $\phi_{\Delta}^e$.
\end{lemma}
\begin{proof}
The statement is obviously local so we may assume that $X$ is the spectrum of a DVR and $Q$ is the unique closed point of $X$.  Write $\Delta = \lambda Q$.  We leave it as an exercise to the reader to check that $\phi_{\Delta}^e$ has $\O_X$ in its image if and only if $\lambda \leq 1$.
\end{proof}

We now explain briefly the relation between $F$-pure and log canonical singularities through several lemmas and corollaries.

\begin{lemma}
\label{lem.BirationalBehaviorOfMaps}
Suppose we are given $\phi : F^e_* \sL \to K(X)$ corresponding to a divisor $\Delta_{\phi}$.  Fix $K_X$ so that $\sL = \O_X( (1-p^e)(K_X + \Delta)) \subseteq K(X)$.  We denote by $\phi : F^e_* K(X) \to K(X)$ the extension to $F^e_* K(X)$ of $\phi$.

Consider a divisorial discrete valuation ring $\O_{X,E} \subseteq K(X)$ centered over $X$.  Then the map
\[
F^e_* \O_X( (1-p^e)(K_X + \Delta)) \otimes \O_{X,E} \to K(X)
\]
corresponds to the divisor $\Delta_{X,E}$ such that $K_{X,E} + \Delta_{X,E} = \pi^*(K_X + \Delta)$.  Here $K_{X,E}$ is the canonical divisor on $\Spec \O_{X,E}$ coming from the uniquely determined ${\bf b}$-divisor we selected implicitly when we fixed $K_X$.
\end{lemma}
\begin{proof}
This is essentially in \cite[Main Theorem]{HaraWatanabeFRegFPure} using different language, other proofs can be found in \cite[Section 7.2]{BlickleSchwedeSurveyPMinusE}.  We briefly sketch the argument.

Choose $\pi : Y \to X$ a proper birational map from a normal $Y$ such that $\O_{X,E}$ appears as the generic point of some divisor $E \subseteq Y$.  We first extend $\phi$ to $F^e_* K(X) = F^e_* K(Y)$, and then consider $\phi_Y = \phi|_{F^e_* \pi^*(\O_X( (1-p^e)(K_X + \Delta)))} : F^e_* \pi^*(\O_X( (1-p^e)(K_X + \Delta))) \to \pi^* K(X) = K(Y)$.  This map $\phi_Y$ gives us a divisor $\Delta_Y$ such that
\[
\O_Y((1-p^e)(K_Y + \Delta_Y)) = \pi^*(\O_X( (1-p^e)(K_X + \Delta)))
\]
 and hence $K_Y + \Delta_Y = \pi^*(K_X + \Delta)$.  However, the divisor $\Delta_Y$ induced from $\phi_Y$ obviously agrees with $\Delta_X$ wherever $\pi$ is an isomorphism.  The result follows by localization.
\end{proof}

We next verify that sub-log canonical and sub-$F$-pure are the same for valuation rings.

\begin{lemma}
Suppose that $R \subseteq K(X)$ is a discrete valuation ring with parameter $t \in R$.  $R \cong L \subseteq K(X)$ and $\phi : F^e_* L \to K(X)$ is any nonzero map inducing a divisor $\Delta_{\phi}$ on $\Spec R$, then $(\Spec R, \Delta_{\phi})$ is sub-log canonical if and only if it is sub-$F$-pure.
\end{lemma}
\begin{proof}
To say that $(\Spec R, \Delta_{\phi})$ is sub-log canonical is simply to assert that $\Delta_{\phi} \leq 1\cdot \Div(t)$.  Now, if $\Phi : F^e_* L \to R$ generates $\Hom_R(F^e_* L, R)$ as an $F^e_* R$-module, then $\phi(F^e_* \blank) = \Phi(F^e_* u t^{a} \blank)$ for some unit $u \in R$ and some $a \in \bZ$.  It follows that $\Delta_{\phi} = {a \over p^e - 1} \Div(t)$.  On the other hand, it is easy to see that $\Image(\phi) \supseteq \O_X$ if and only if $a \leq p^e - 1$.  This proves the lemma.
\end{proof}

Combining the previous two results, we immediately have the following.  This was first shown in \cite[Main Theorem]{HaraWatanabeFRegFPure} although they assumed that $\Delta$ is effective.
\begin{corollary}
\label{cor.SubFPureImpliesSubLC}
If $(X, \Delta)$ is a pair with $(p^e - 1)(K_X + \Delta)$ $\bQ$-Cartier, and $(X, \Delta)$ is sub-$F$-pure, then $(X, \Delta)$ is sub-log canonical.
\end{corollary}
\begin{proof}
This is easy, if $\phi : F^e_* \sL \to K(X)$ corresponds to $\Delta$ and $\phi(F^e_* \sL) \supseteq \O_X$, then, abusing notation and extending $\phi$ to the fraction field, obviously $1 \in \phi(F^e_* \sL \otimes_{\O_X} \O_{X,E})$ and hence $\O_{X,E} \subseteq \phi(F^e_* \sL \otimes_{\O_X} \O_{X,E})$ for any divisorial discrete valuation ring $\O_{X,E}$ lying over $X$.  In particular, each $(\Spec \O_{X,E}, \Delta_{X,E})$ is sub-$F$-pure (where $K_{X,E} + \Delta_{X,E} = \pi^*(K_X + \Delta)$).  The previous two lemmas immediately imply that the discrepancy divisor along all such valuations is $\geq -1$ and the result follows.
\end{proof}

\subsection{Global $F$-splittings}

Throughout this subsection, we suppose that $(X, \Delta)$ is a pair and $(p^e - 1)(K_X + \Delta)$ is Cartier.

\begin{definition}[Global $F$-splitting]
\label{def.GlobalFSplitting}
If $\Delta \geq 0$, we say that \emph{$(X, \Delta)$ is globally F-split} if the map induced by $\Delta$ is surjective on global sections
\[
H^0(X, F^e_* \O_X( (1-p^e)(K_X + \Delta))) \twoheadrightarrow H^0(X, \O_X) \quad \mbox{for some } e>0.
\]
Note that this implies that $(X, \Delta)$ is $F$-pure and hence sub-$F$-pure.

For an arbitrary $\Delta$, we say that \emph{$(X, \Delta)$ is globally sub-$F$-split} if image of
\[
H^0(X, F^e_* \O_X( (1-p^e)(K_X + \Delta))) \to H^0(X, K(X)) \quad \mbox{for some } e>0
\]
contains $1$ (and hence its image globally generates a rank-$1$-module containing $\O_X$, and thus $(X, \Delta)$ is sub-$F$-pure).
\end{definition}

We will frequently be interested in splitting relatively, see \cite{SchwedeSmithLogFanoVsGloballyFRegular,HaconXuThreeDimensionalMinimalModel}.

\begin{definition}[Relative $F$-splitting]
Suppose that $\pi : X \to Z$ is a map of schemes.  We say that a pair \emph{$(X, \Delta)$ is globally F-split relative to $\pi$ (or over $Z$)} if there is an open cover $\{U_i\}$ over $Z$ such that $(\pi^{-1} U_i, \Delta|_{\pi^{-1} U_i})$ is globally $F$-split for each $i$.

Likewise we say that \emph{$(X, \Delta)$ is globally sub-$F$-split relative to $\pi$} if there is an open cover $\{U_i\}$ over $Z$ such that $(\pi^{-1} U_i, \Delta|_{\pi^{-1} U_i})$ is globally sub-$F$-split for each $i$.
\end{definition}

\subsection{Log canonical and $F$-pure centers}

We recall certain distinguished subvarieties of log canonical (respectively $F$-pure) pairs, the LC centers (respectively $F$-pure centers).

\begin{definition}[LC centers]
\label{def.LCCenter}
Suppose that $(X, \Delta)$ is a pair.  We say that an integral subscheme \emph{$Z \subseteq X$ is a log canonical center of $(X, \Delta)$} if
\begin{itemize}
\item[(a)] $\Delta$ is effective at the generic point $Q$ of $Z$,
\item[(b)] $(X, \Delta)$ is log canonical at the generic point $Q$ of $Z$,
\item[(c)] There exists a proper birational map $\pi : Y \to X$ from a normal $Y$ such that if we write $K_Y + \Delta_Y = \pi^*(K_X + \Delta_X)$, then there exists a prime divisor $E$ on $Y$ such that $\coeff_E(\Delta_Y) = 1$ and $\pi(E) = Z$.
\end{itemize}
\end{definition}

\begin{definition}[$F$-pure centers]
\label{def.FPureCenter}
Suppose that $(X, \Delta)$ is a pair in characteristic $p > 0$ with $(1-p^e)(K_X + \Delta)$ Cartier.  Say that $\Delta$ corresponds to some map $\phi : F^e_* \sL \to K(X)$, where $\sL = \O_X( (1-p^e)(K_X + \Delta))$.  We say that an integral subscheme \emph{$Z \subseteq X$ is an $F$-pure center of $(X, \Delta)$} if
\begin{itemize}
\item[(a)] $\Delta$ is effective at the generic point $Q$ of $Z$,
\item[(b)] $(X, \Delta)$ is $F$-pure at the generic point $Q$ of $Z$,
\item[(c)] If $\O_{X,Q}$ is the stalk at the generic point $Q$ of $Z$ and $I_Z$ is the ideal defining $Z$, then $\phi(F^e_* I_Z \cdot \sL) \cdot \O_{X,Q} \subseteq I_Z \cdot \O_{X,Q} \subseteq K(X)$.
\end{itemize}
We note that the choice of map $\phi$ does not change whether or not a particular $Z$ is an $F$-pure center since different maps correspond to the same divisor if and only if the maps are the same up to multiplication by a unit.
\end{definition}

We recall the following result about $F$-pure centers, see \cite[Lemmas 3.6 and 3.7]{SchwedeCentersOfFPurity} for the case when $\Delta$ is effective.

\begin{lemma}
Being an $F$-pure center is a local condition. In other words, $Z$ is an $F$-pure center of $(X, \Delta)$ if and only if $V(Q\O_{X,Q}) \subseteq \Spec \O_{X,Q}$ is an $F$-pure center of  $(\Spec \O_{X,Q}, \Delta|_{\Spec \O_{X,Q}})$, where $Q$ is the generic point of $Z$.
\end{lemma}
\begin{proof}
Conditions (a) and (b) of Definition \ref{def.FPureCenter} are certainly equivalent for either case.  Now, if $Z$ is an $F$-pure center of $(X, \Delta)$, then $V(Q \O_{X,Q})$ is also an $F$-pure center by localization.  On the other hand 
\[
\phi\big(F^e_* (I_Z \cdot \sL \cdot \O_{X,Q})\big) \subseteq I_Z \cdot \O_{X,Q}
\]
is the same as
\[
\phi(F^e_* I_Z \cdot \sL) \cdot \O_{X,Q} \subseteq I_Z \cdot \O_{X,Q}
\]
and the proof is complete.
\end{proof}

\begin{remark}
The above definition differs in several ways from the one given in \cite{SchwedeCentersOfFPurity}.  One notable way is that we do not assume that $\Delta \geq 0$ (we also make several simplifying assumptions).
\end{remark}

\begin{remark}
Part (c) for either definition (\ref{def.LCCenter} and \ref{def.FPureCenter}) can also be interpreted as follows (also assuming resolution of singularities for \ref{def.LCCenter}):  For any Cartier divisor $H > 0$ passing through the generic point $Q$ of $Z$, $(X, \Delta+tH)$ is not sub-log canonical at $Q$ (respectively not sub-$F$-pure at $Q$) for any $t > 0$.
\end{remark}

\begin{lemma}
Suppose that $X$ is an $F$-finite normal integral scheme of characteristic $p > 0$ and that $\Delta$ is a $\bQ$-divisor such that $(p^e - 1)(K_X + \Delta)$ is Cartier and $\Delta$ satisfies conditions (a) and (b) from \autoref{def.FPureCenter}.  Then if $Z \subseteq X$ is a log canonical center of $(X, \Delta)$, then $Z$ is also an $F$-pure center.
\end{lemma}
\begin{proof}
In the case that $\Delta \geq 0$, this follows from \cite{SchwedeCentersOfFPurity}.  However, by localizing at the generic point of $Z$, we may assume that $\Delta \geq 0$.
\end{proof}

\subsection{The different and divisorial part thereof}
\begin{definition}\textnormal{\cite[Chapter 4]{Kol13}}
	Let $(X, S+\Delta)$ be a log pair, i.e., $X$ is a normal integral scheme, $K_X+S+\Delta$ is $\QQ$-Cartier, $S$ is a reduced divisor and $\Delta$ is a $\bQ$-divisor with no common components with $S$. Let $S^{\textnormal{N}} \to S$ be the normalization of $S$. There exists a canonically determined $\QQ$-divisor $\Delta_{S^{\textnormal{N}}}$ on $S^{\textnormal{N}}$ such that $(K_X+S+\Delta)|_{S^{\textnormal{N}}}\sim_{\QQ} K_{S^{\textnormal{N}}}+\Delta_{S^{\textnormal{N}}}$. The $\Q$-divisor $\Delta_{S^{\textnormal{N}}}$ is called the \emph{different of the adjunction} or simply the \emph{different}.
	\end{definition}
If $\Delta\>0$, then $\Delta_{S^{\textnormal{N}}}\>0$. For this and other properties of the different see \cite[Chapter 17]{KollarFlipsAndAbundance} and \cite[Chapter 4]{Kol13}.

\begin{definition}\label{dfn:div-part}
		Let $f: X\to Z$ be a surjective proper morphism between two normal integral schemes and suppose $K_X+\Delta\sim_\Q f^*L$, where $\Delta=\sum d_i\Delta_i$ is a $\bQ$-divisor on $X$, and $L$ is a $\Q$-Cartier $\Q$-divisor on $Z$. Suppose $(X, \Delta)$ is sub-log canonical over a neighborhood of the generic point of $Z$, i.e., $(f^{-1}U, \Delta|_{f^{-1}U})$ is sub-log canonical for some Zariski dense open subset $U\subset Z$. Then we define two divisors $\Delta_{\div}$ and $\Delta_{\mod}$ on $Z$ in the following way:
	\begin{align*}
	\Delta_{Z, \div} = \Delta_{\div} &=\sum (1-c_Q)Q, \text{ where } Q\subset Z \text{ are prime Weil divisors of } Z, \text{ and}\\
	c_Q &=\text{sup}\left\{ c\in \R \;\Big|\; \begin{array}{r}(X, \Delta+cf^*(\eta_Q)) \text{ is sub-log canonical over a} \\\text{neighborhood of the generic point } \eta_Q \text{ of } Q  \end{array}\right\},\\
	\Delta_{Z, \mod} = \Delta_{\mod} &=L-K_Z-\Delta_{\div}, \text{ so that } K_X+\Delta\sim_\Q f^*(K_Z+\Delta_{\div}+\Delta_{\mod}).
	\end{align*}
The divisor $\Delta_{\div}$ is called the \emph{divisorial part of the adjunction} and $\Delta_{\mod}$ the \emph{moduli part}.	
	\end{definition}
\begin{remark}[Is $\Delta_{\div}$ an honest divisor?]
\label{rem.HonestDivisors}
We do not know whether $\Delta_{\div}$ is an honest divisor in general (the right hand side of $\Delta_{\div}$ could be an infinite sum). However if we are working with varieties in characteristic $0$, it is not hard to see that $\Delta_{\div}$ is in fact a finite sum. Indeed, by replacing $X$ with a log resolution of $(X, \Delta)$ we may assume that $(X, \Delta)$ is a SNC pair (since obviously the definition of $\Delta_{\div}$ is independent of any birational modification of $X$). Write $\Delta=\Delta^h+\Delta^v$, where each component of $\Delta^h$ dominates $Z$ and no component of $\Delta^v$ dominates $Z$. Let $\Delta_i$, an integral scheme, be a component of $\Delta^h$.

Consider the induced morphism $f|_{\Delta_i}:\Delta_i\to Z$. By generic smoothness, there exists a dense open set $U_i\subseteq Z$ such that the fibers of $f|_{\Delta_i}$ over $U_i$ are all smooth. Let $Q$ be a prime Weil divisor on $Z$ whose generic point $\eta_Q$ is contained in $U_i$. Then $(f|_{\Delta_i})^*(\eta_Q)$ is smooth over $k(\eta_Q)$. But $(f|_{\Delta_i})^*(\eta_Q)$ is the scheme theoretic intersection of $\Delta_i$ and $f^*(\eta_Q)$. Therefore $\Delta_i$ intersects $f^*(\eta_Q)$ transversally. Let $U$ be the intersection of all such open sets $U_i$ corresponding to all components $\Delta_i$ of $\Delta^h$, and $V=U\setminus f(\mbox{Supp}(\Delta^v))$. Then for any prime Weil divisor $Q$ of $Z$ such that $\eta_Q\in V$, $\Delta+f^*(\eta_Q)$ is a SNC divisor; in particular $(X, \Delta+f^*(\eta_Q))$ is LC over $\eta_Q$. Therefore $c_Q=1$ for all $\eta_Q\in V$, and the right hand side of $\Delta_{\div}$ is a finite sum.

This proof fails in characteristic $p>0$ because resolutions of singularities are not known to exist and even worse, generic smoothness is known to fail. However with an additional hypothesis (which appears frequently for us, see \autoref{prop.PInverseOnFields} and the results immediately following it) we can show that $\Delta_{\div}$ is a finite sum in char $p>0$.
\end{remark}

\begin{lemma}
\label{lem.DivisorialPartOfDifferent}
With the same hypothesis as \autoref{dfn:div-part}, assume additionally that all schemes are $F$-finite and of characteristic $p > 0$.  Further assume that $(p^e - 1) L$ and $(p^e - 1)(K_X + \Delta)$ are Cartier and that $(1-p^e)(K_X + \Delta) \sim (1-p^e) f^* L$.  Fix a map $$\phi_{\Delta} : F^e_* \O_X( (1-p^e) f^* L) \to K(X)$$ corresponding to $\Delta$ as in \autoref{subSec.MapsAndNonEffective}.  Now assume that\footnote{Here $f^{-1}$ is inverse image of sheaves in the category of Abelian groups, as in \cite[Chapter II, Section 1]{Hartshorne}} \begin{equation}
\label{eq.MapSendsKZtoKZ}
0 \neq \phi_{\Delta}(F^e_* f^{-1} \O_Z( (1-p^e) L) ) \subseteq f^{-1} K(Z) \subseteq K(X).
\end{equation}
Let $\Delta_Z$ denote the divisor on $Z$ corresponding to the induced $\phi_Z : F^e_* \O_Z( (1-p^e) L) \to K(Z)$.
Then $\Delta_{\div}$ is an honest $\bR$-divisor and $\Delta_{\div} \leq \Delta_Z$.
\end{lemma}
While we phrased this in a somewhat funny way involving $f^{-1}$ and $L$, the requirement
\[
0 \neq \phi_{\Delta}(F^e_* f^{-1} \O_Z( (1-p^e) L) ) \subseteq f^{-1} K(Z)
 \]
of \autoref{lem.DivisorialPartOfDifferent} should be thought of as requiring that $\phi_{\Delta} : F^e_* K(X) \to K(X)$ restricts to a nonzero and hence surjective map $F^e_* K(Z) \to K(Z)$. We used $f^{-1}$ and not $f^*$ because we do not want to tensor up to sheaves of $\O_X$-modules.
	
\begin{proof}
First we show that $\Delta_{\div} \leq \Delta_Z$.  This is local on $Z$ so we assume that $Z$ is the spectrum of a DVR with parameter $x$ and that $z \in Z$ is the unique closed point.  For any rational $\lambda$, we consider the condition that $(Z, \Delta + \lambda z)$ is sub-$F$-pure, which means that
\[
\phi_Z : F^e_* x^{\lceil(p^e - 1)\lambda \rceil} \cdot \O_Z( (1-p^e) L) \to K(Z)
\]
has $1$ in its image for sufficiently divisible $e > 0$.  But then since $\phi_Z$ extends to $\phi_{\Delta}$ we see that $\phi_{\Delta} : F^e_* x^{\lceil(p^e - 1)\lambda \rceil} \O_X( (1-p^e) f^* L) \to K(X)$ also has $1$ in its image.  Hence $(X, \Delta + \lambda f^* z)$ is also sub-$F$-pure and thus also sub-log canonical by \autoref{cor.SubFPureImpliesSubLC}.  It follows that
\[
\begin{array}{rl}
c_z := & \sup\{ c \in \bR \; |\; (X, \Delta + c f^* z) \text{ is sub-log canonical}\}\\
 \geq & \sup\{ c \in \bR \; |\; (Z, \Delta_Z + c z) \text{ is sub-$F$-pure}\}.
\end{array}
\]
If we write $\Delta_Z = a z$, then $\sup\{ c \in \bR \; |\; (Z, \Delta_Z + c z) \text{ is sub-$F$-pure}\} = 1-a$.  Hence $c_z \geq 1-a$ and so $1-c_z \leq a$.  This proves the inequality $\Delta_{\div} \leq \Delta_Z$.

For the statement that $\Delta_{\Div}$ is actually a divisor we can no longer assume that $Z$ is the spectrum of a DVR.  Since $\Delta_{\div} \leq \Delta_Z$, we just need to show that $\Delta_{\Div}$ can have at most finitely many components with negative coefficients.
\begin{claim}
If $Q$ is a point of codimension $1$ in $Z$ and $\coeff_Q(\Delta_{\Div}) < 0$, then $\Delta^v$ has a component whose generic point maps to $Q$.
\end{claim}
\begin{proof}[Proof of claim]
By hypothesis, there exists a real number $c > 1$ such that $(X, \Delta + c f^*(\eta_Q))$ is sub-log canonical over a neighborhood of $Q$.    Let $E$ be an irreducible component in $\Supp(f^*(\eta_Q))$.  Since $(X, \Delta + cf^*(\eta_Q))$ is sub-log canonical, $1 \geq \coeff_E(\Delta + c f^*(\eta_Q)) \geq \coeff_E(\Delta) +c$, then $\coeff_E(\Delta) \< 1-c < 0$.
This proves the claim.
\end{proof}

The claim proves the lemma since $\Delta$ has only finitely many components.
\end{proof}



Now suppose that $W \subseteq X$ is a log canonical center of $(X, \Delta)$.  For any proper birational map $W' \to W$ with $W'$ normal, we can find a proper birational map $\pi : X' \to X$ such that there exists a divisor $E \subseteq X$ with:
\begin{enumerate}
\item $\pi(E) = W$
\item $\pi|_E$ factors through $W' \to W$
\item the discrepancy along $E$ is equal to -1.
\end{enumerate}

\begin{definition}
\label{defn.DivisorialPartOfTheDifferent}
The \emph{divisorial part of the different on $W' \to W \subseteq X$, relative to $E$}, denoted by $\Delta_{W', E, \Div}$ (or simply by $\Delta_{W', \Div}$ if $E$ is implicit), is defined to be the divisorial part of the adjunction of $E \to W'$ defined as in \autoref{dfn:div-part}.
\end{definition}

\begin{remark}
While it is clear that $\Delta_{W', E, \Div}$ only depends on the valuation of $E$ (and not the particular choice of $X'$), we do not know that it is independent of the choice of $E \subseteq X'$ in general.  It is because of this that we include $E$ in the definition and notation for the divisorial part of the different.   Note that in characteristic $p > 5$, if $\dim X \leq 3$, then by \cite[Lemma 4.10]{DasHacon} we know that  $\Delta_{W', E, \Div}$ is independent of the choice of $E$ in certain cases.  We also believe this is true in higher dimensions in characteristic zero by a similar method.  Unfortunately we do not know of a reference and we feel that proving it here would take us away from our focus (characteristic $p > 0$).  On the other hand we will see in \autoref{sec.CanBundleFormula} that the $F$-different can be deduced via a version of a canonical bundle formula for $E \to W'$, see \autoref{cor.FadjunctionViaCanBundle}.  However, because it can also be obtained independently from $E$, the $F$-different is independent of the choice of $E$.
\end{remark}

\begin{remark}
\label{rem.DivPartBDivisor}
	
	Let $\sigma:Z'\to Z$ be a birational morphism with $Z'$ normal and $X'$ the normalization of the component of $X\times_Z Z'$ dominating $Z'$ so that $X' \to X$ is birational.
	\begin{equation}\label{base-change-diagram}
    \begin{array}{c}
	\xymatrixcolsep{3pc}\xymatrix{ X\ar[d]_f & X'\ar[l]_{\sigma'}\ar[d]^{f'}\\
	Z & Z'\ar[l]^{\sigma}
	}
\end{array}
	\end{equation}	
	We can then define the divisorial part of the adjunction $\Delta_{Z', \Div}$ for $X' \to Z'$.  It is clear that $\sigma_* \Delta_{Z', \Div} = \Delta_{Z, \Div}$ since at points where $\sigma$ is an isomorphism, there is nothing to do.  Therefore we can define ${\bf \Delta}_{\Div}$ the {\bf b}-divisor over $Z$ picking out $\Delta_{Z', \Div}$.
	\end{remark}

We finally discuss a notion of descent.

\begin{definition}
Suppose that $Z$ is a normal integral scheme and ${\bf D}$ is a {\bf b}-divisor over $Z$.  For a proper birational map $Z' \to Z$, we say that \emph{${\bf D}$ descends to $Z'$ (as a $\bQ$-Cartier divisor)} if ${\bf D}_{Z'}$ is $\bQ$-Cartier and for every further birational $\mu : Z'' \to Z'$, $\mu^* {\bf D}_{Z'} = {\bf D}_{Z''}$.
\end{definition}

And we recall the following result of Ambro.

\begin{lemma}\textnormal{\cite{AmbroAdjunctionConjecture,HaconLCInversionAdjunction}}  Suppose that $Z$ is an log canonical center of $(X, \Delta)$ relative to $E$ and that $X$ is a variety of characteristic zero.  Then there exists a proper birational map $Z' \to Z$ with $Z'$ normal such that the {\bf b}-divisor ${\bf K}_Z + {\bf \Delta}_{\Div}$ descends to $Z'$ as a $\bQ$-Cartier divisor.
\end{lemma}

\subsection{The $F$-different}
\label{subsec.Fdifferent}

Now we define the $F$-different, the primary object of study in this paper.  Suppose that $(X, \Delta)$ is a pair with $(p^e - 1)(K_X + \Delta)$ Cartier.  Suppose that $W \subseteq X$ is an $F$-pure center and $\mu : W^{\textnormal{N}} \to W$ is its normalization.  Set $\sL = \O_X( (1-p^e)(K_X + \Delta))$ and let $\phi_W : F^e_* \sL|_W \to \O_W(G|_W) \subseteq K(W)$ be the map induced from the following diagram.  Here $G$ is some effective Cartier divisor not containing the generic point of $W$ chosen to contain the image of $\phi_{\Delta}$
\begin{equation}
\label{eq.AdjunctionInducingFDiff}
\xymatrix{
0 \ar[d] & 0 \ar[d] \\
F^e_* I_W \cdot \sL \ar[r]^{\phi_{\Delta}} \ar[d] & I_W \cdot \O_X(G)\ar[d]\\
F^e_* \sL \ar[d] \ar[r]^{\phi_{\Delta}} & \O_X(G) \ar[d] \ar@{^{(}->}[r] & K(X) \\
F^e_* \sL|_W \ar[d] \ar[r] & \O_W(G|_W) \ar[d] \ar@{^{(}->}[r] & K(W) \\
0 & 0 \\
}
\end{equation}
Any map $F^e_* \sL|_W \to K(W)$ clearly extends to a map $\phi_{W^{\textnormal{N}}} : F^e_* \mu^*(\sL|_W) \to K(W^{\textnormal{N}})$ (see \cite{MillerSchwedeSLCvFP} for more discussion on the divisorial interpretation of this extension).
\begin{definition}[The $F$-different]
\label{defn.FDifferent}
The \emph{$F$-different of $(X, \Delta)$ on $W^{\textnormal{N}}$} is defined to be the divisor $\Delta_{W^{\textnormal N},\Fdiff} := \Delta_{W^{\textnormal N}}$ corresponding to $\phi_{W^{\textnormal N}}$.  More generally, for any birational map $\kappa : W' \to W$ from a normal $W'$, we have an induced map $\phi_{W'} : F^e_* \kappa^*(\sL_W) \to K(W')$ and hence a divisor $\Delta_{W'} = \Delta_{W', \Fdiff} = \Delta_{W', \Fdiff}$.  This gives us the $F$-different as a {\bf b}-divisor, denoted ${\bf \Delta}_{\Fdiff}$.
\end{definition}

\begin{remark}
\label{rem.FDiffBDivisor}
Even though the $F$-different is a {\bf b}-divisor, the sum of the two {\bf b}-divisors ${\bf K} + {\bf \Delta}_{\Fdiff}$ is in fact $\bQ$-Cartier and descends down to any $W'$.  Indeed for any birational $\rho : W'' \to W'$ between normal integral schemes dominating $W$, we have $\rho^*(K_{W'} + \Delta_{W', \Fdiff}) = K_{W''} + \Delta_{W'', \Fdiff}$. In summary ${\bf K} + {\bf \Delta}_{\Fdiff}$ descends all the way to $W^N$.
\end{remark}

\section{A brief description of computing the $F$-different and an example}

We begin by a description of an algorithm for computing the $F$-different in explicit examples.

\begin{lemma}
\label{lem.ComputationalFDif}
Suppose that $Y = \bA^n = \Spec S = \Spec k[x_1, \ldots, x_n]$ and $\Spec S/I = X \subseteq Y$ is a closed normal subscheme.    Fix $\Delta \geq 0$ a $\bQ$-divisor on $X$ such that $(p^e-1)(K_X + \Delta)$ is a Cartier divisor for some positive integer $e>0$.  Suppose $Z = \Spec S/J \subseteq X$ is an $F$-pure center of $(X, \Delta)$.  Then
\begin{enumerate}
\item $\Delta$ induces a (non-unique) element $f \in (I^{[p^e]} : I), f \notin (I^{[p^e]})$.
\item $F^e_* f \in F^e_* {J^{[p^e]} : J \over J^{[p^e]}}$ corresponds to a global section of $F^e_* \O_Z( (1-p^e)K_Z)$ via the following isomorphism
 \[ F^e_* {J^{[p^e]} : J \over J^{[p^e]}}\cong \Hom(F^e_* S/J, S/J) \cong \mbox{H}^0(Z, F^e_* \O_Z( (1-p^e)K_Z)).\]
\label{lem.ComputationalFDif.2}
\end{enumerate}
Using the isomorphism of (b), $F^e_* f$ yields an induced section $\overline{f} \in \mbox{H}^0(Z, \O_Z( (1-p^e)K_Z)) = \omega_{R/I}^{(1-p^e)}$ whose corresponding divisor is the $F$-different $\Delta_{Z,\Fdiff}$.

Furthermore, suppose that $(p^e-1)K_Z \sim 0$ and so write $J^{[p^e]} : J = \langle g_e \rangle + J^{[p^e]}$.  Then $f = h g_e + g$ with $g \in J^{[p^e]}$.  In this case the $F$-different is ${1 \over p^e - 1} \Div_Z(h)$.
\end{lemma}
\begin{proof}
The fact that $\Delta$ corresponds to $f$ is simply Fedder's criterion \cite[Lemma 1.6]{FedderFPureRat}.  This $f$ induces a divisor $\Delta_Y$ on $Y$ such that $(X, \Delta)$ is an $F$-pure center of $(Y, \Delta_Y)$.  For \autoref{lem.ComputationalFDif.2}, we use the Fedder-type criterion for $F$-pure centers \cite[Proposition 3.11]{SchwedeCentersOfFPurity}.  For the final statement, we observe that $g_e$ corresponds to the generating homomorphism $\Phi_Z$ of $\Hom(F^e_* S/J, S/J)$, see \cite{SchwedeFAdjunction}.
We see that $f$ corresponds to pre-multiplying the generating homomorphism $\Phi_Z$ by $h$.  The statement about the different follows.
\end{proof}

The following example was first worked out in 2010 by the second author, David Speyer and Chenyang Xu.

\begin{example}
\label{ex.EllipticCurveExample}
Suppose that $f:X \to \Spec k[t] = \bA^1$ is a family of cones over elliptic curves defined by $a = zy^2 - x(x-z)(x-tz)$ with a section $\gamma : \bA^1 \to X$ mapping to the cone points.  Further consider the log resolution $\pi : \tld X \to X$ which is obtained by blowing up the image of $\gamma$ (which we now call $Z$).  Finally note that $X$ is $F$-pure at the generic point of $Z$ by Fedder's criterion \cite{FedderFPureRat}.

Note $X$ is Gorenstein since it is a hypersurface, and set $\Delta = 0$.
It then follows that $Z$ is a log canonical center of $(X, \Delta)$ and hence an $F$-pure center of $(X, \Delta)$.

\begin{claim}\label{clm:F-different-locus}
	The $F$-different $\Delta_Z$ on $Z=V(x, y, z)\subset X$ is supported at exactly those points $(0, 0, 0, s)\in Z$ such that the elliptic curve corresponding to $f^{-1}(s)$ is supersingular, where $f:X\to\Spec k[t]$.
	\end{claim}

 We will prove this claim by two different methods. First by direct computation of polynomials.
\begin{proof}[Proof \#1 of \autoref{clm:F-different-locus}]
We write
\[
a^{p-1} = h(t)x^{p-1}y^{p-1}z^{p-1} + G(x,y,z,t)
\]
where $G(x,y,z,t) \in \langle x^p, y^p, z^p \rangle$.  This gives us an alternate proof that $Z = V(x,y,z)$ is an $F$-pure center again by \cite[Proposition 3.11]{SchwedeCentersOfFPurity}.  We first compute the $F$-different via \autoref{lem.ComputationalFDif}.  Note that the map $F_* \O_X \to \O_X$ corresponding to $\Delta$ lifts to a map $\phi_{\bA^4} : F_* k[x,y,z,t] \to k[x,y,z,t]$ compatible with $X \subseteq \bA^4 = \Spec k[x,y,z,t]$, namely take the generating map $\Phi \in \Hom_{k[x,y,z,t]}(F_* k[x,y,z,t], k[x,y,z,t])$ on $\bA^4$ and pre-multiply by $a^{p-1} = f$.  We notice that $g_1 = x^{p-1}y^{p-1}z^{p-1}$ (here $g_1 = g_e$ is defined as in the statement of \autoref{lem.ComputationalFDif}).  Thus the $F$-different is simply ${1 \over p - 1} \Div_Z(h(t))$ by \autoref{lem.ComputationalFDif}.  Note that $h(t)$ is the Hasse polynomial, which vanishes exactly at those $t$ values where the associated elliptic curve is supersingular.
\end{proof}

\begin{proof}[Proof \#2 of \autoref{clm:F-different-locus}]
Now we study the $F$-different via geometry, in a way which will be similar to what we will do in later sections.  Note that $(X, \Div_X(t - \lambda))$ has a log canonical center at $Q = (x, y, z, t-\lambda)$.   Furthermore, by blowing up $\tld X$ at the inverse image $C_{\lambda}$ of that point $Q$, one obtains a log resolution $\mu : X' \to X$ with two exceptional divisors, $E_1$ dominating $Z$ and $E_2$ dominating $Q$.  Both of these exceptional divisors have discrepancy $-1$ with respect to the pair $(X, 1 \cdot \Div_X(t - \lambda))$.
\begin{claim}
\label{clm.FpureImpliesFSplit}
If $(X,1 \cdot \Div_X(t - \lambda))$ is $F$-pure, the exceptional divisor $E_2$ is $F$-split.
\end{claim}
\begin{proof}
Write $\phi_X : F^e_* \O_X\big( (1-p^e)(K_X + \Div_X(t - \lambda))\big) \to \O_X$ corresponding to the divisor $1 \cdot \Div_X(t - \lambda)$.  This map is surjective by hypothesis.  Since all the discrepancies on $\tld X$ are non-positive, this map extends to a map
\[
\phi_{X'} : F^e_* \O_{X'}\big( \mu^*((1-p^e)(K_{X} + \Div_X(t-\lambda)))\big) \to \O_{X'}.
  \]
Note $\mu_*$ of $\phi_{X'}$ is $\phi_X$ and so $\phi_{X'}$ is also surjective, even on global sections (since $1$ is in the image along the global sections).  Since $E_2$ has discrepancy $-1$, $\phi_{X'}$ is compatible with $E_2$ and so by the argument and diagram of \autoref{eq.AdjunctionInducingFDiff}, we obtain a map
\[
\phi_{E_2} : F^e_* \O_{E_2} \cong \O_{X'}( \mu^*((1-p^e) (K_{X} + \Div_X(t-\lambda))))|_{E_2} \to \O_{E_2}
 \]
by restriction.  This map also has $1$ in its image among global sections, and hence $E_2$ is Frobenius split.  This proves \autoref{clm.FpureImpliesFSplit}.
\end{proof}
The \autoref{clm.FpureImpliesFSplit} implies that the associated elliptic curve is also $F$-split since $E_2$ maps onto the elliptic curve $\rho : E_2 \to C_{\lambda}$ with $\rho_* \O_{E_2} = \O_{C_{\lambda}}$, see \cite{MehtaRamanathanFrobeniusSplittingAndCohomologyVanishing,BrionKumarFrobeniusSplitting}.

By $F$-adjunction $(X, 1\cdot \Div_X(t - \lambda))$ is $F$-pure (in a neighborhood of $Z$) if and only if $(Z, \Delta_Z + \Div_Z(t - \lambda))$ is $F$-pure.
The latter is $F$-pure at the closed point $V(t - \lambda) \subseteq Z$ if and only if $\Delta_Z$ does not have $\Div_Z(t - \lambda)$ in its support.
Putting this together, if $\Delta_Z$ does not have $\Div_Z(t-\lambda)$ in its support, then $(X, 1\cdot \Div_X(t - \lambda))$ is $F$-pure, which implies that $C_{\lambda}$ is $F$-split as we already saw.
In other words if $\lambda$ corresponds to a supersingular elliptic curve $C_{\lambda}$, then $\Delta_Z$ must have $\Div_Z(t - \lambda)$ among its components.

Conversely, suppose that $\lambda$ corresponds to an ordinary elliptic curve $E_{\lambda}$.  The generating map (ie the map corresponding to the dual of Frobenius) on the associated elliptic curve $\psi : F^e_* \omega_{E_{\lambda}} \cong F_* \O_{E_{\lambda}} \to \O_{E_{\lambda}} \cong \omega_{E_{\lambda}}$ is always the map induced by the pair $(X, \Div_X(t - \lambda))$ on $\tld X$ as above.  On the global sections of the elliptic curve, the map $\psi$ sends units to units.  This implies the map associated to $(X, \Div_X(t - \lambda))$, when extended to $\tld X$, has to send units to non-zero elements which restrict to units on $E_{\lambda}$.  Thus back on $X$, units must be sent to elements that are units near $Z$ and the proof of the second proof \autoref{clm:F-different-locus} is complete.
\end{proof}

In conclusion, the $F$-different can exhibit some quite complicated behavior as in this case it picks out the supersingular elliptic curves.
\end{example}

\section{First properties of $F$-pure centers and the $F$-different}

$F$-pure centers and log canonical places above them, can have very nice properties.
Suppose that $(X, \Delta)$ is a sub-log canonical pair in characteristic $p > 0$ with $(p^e - 1)(K_X + \Delta) \sim 0$.  Suppose further that $\pi : X' \to X$ is a proper birational map from a normal variety $X'$ and suppose that $\pi^*(K_X + \Delta) = K_{X'} + \Delta_{X'}$.  Suppose that $E \subseteq X'$ is a prime (usually exceptional) divisor of discrepancy $-1$ dominating an integral  subscheme $W \subseteq X$ such that $(X, \Delta)$ is $F$-pure at the generic point of $W$, which implies that $\Delta$ is effective at the generic point of $W$ (and hence that $W$ is an $F$-pure center and log canonical center).  Let $\phi_{\Delta} : F^e_* \O_X( (1-p^e)(K_X + \Delta)) \cong F^e_* \O_X \to K(X)$ be the map corresponding to $\Delta$ and we abuse notation and, after fixing $K_X$ as an honest divisor, denote by $\phi : F^e_* K(X) \to K(X)$ the induced map on the fraction field.  We use $\phi_{\Delta_{X'}}$ to be the map $F^e_* \O_{X'}( (1-p^e)\pi^* (K_X + \Delta)) \cong F^e_* \O_{X'}( (1-p^e)(K_{X'} + \Delta_{X'})) \to K(X')$ which generically agrees with $\phi$.  We use $\phi_E$ and $\phi_W$ to be the maps induced by restricting $\phi_{\Delta_{X'}}$ and $\phi_{\Delta}$ respectively to $E$ and $W$ as in \autoref{subsec.Fdifferent}.

\begin{proposition}
\label{prop.PInverseOnFields}
With notation as above, the map $\phi$ induces nonzero horizontal maps in the commutative diagram below.
\[
\xymatrix{
F^e_* K(E) \ar[r]^{\phi_E} & K(E) \\
 F^e_* K(W) \ar@{^{(}->}[u]\ar[r]_{\phi_W} & K(W)\ar@{^{(}->}[u]
}
\]
Furthermore the vertical maps are simply induced by $\pi_E : E \to W$.
\end{proposition}
\begin{proof}
We may assume that $X = \Spec R$ is affine and that $W = V(I) = \Spec R/I$.  We may also assume that $\Delta \geq 0$.
Choose an open affine neighborhood $U \subseteq X'$ containing the generic point of $E$ where $\Delta_{X'}|_U \geq 0$ and $(p^e - 1)(K_{X'} + \Delta_{X'})|_U \sim 0$.  Note then that the different $\Delta_E$, satisfying $(K_{X'} + \Delta_{X'})|_E \sim_{\bQ} K_E + \Delta_E$, is effective (that is, $\Delta_E|_{U \cap E} \geq 0$).  We may write $U = \Spec R'$ and $E = V(J) = \Spec(R'/J)$.

We then have a map $\alpha : R \hookrightarrow R'$ with $R \cap J = I$.  Let $\phi_R : F^e_* R \to R$ be the map obtained by taking global sections of $\phi_{\Delta}$.  Now, since $R$ and $R'$ are birational, and we have $\Delta_{X'}|_U \geq 0$, we obtain a map $\phi_{R'} : F^e_* R' \to R'$, corresponding to $\Delta_{X'}$, such that the following diagram commutes:
\[
\xymatrix{
F^e_* R \ar@{^{(}->}[d]_{F^e_* \alpha} \ar[r]^{\phi_R} & R \ar@{^{(}->}[d]^{\alpha} \\
F^e_* R' \ar[r]_{\phi_{R'}} & R'.
}
\]
Since $E$ has discrepancy $-1$ in $X'$, i.e., $\coeff_E \Delta_{X'} = 1$, we see that $J$ is compatible with $\phi_{R'}$, and we already know that $I$ is compatible with $\phi_R$, although it also follows from the diagram above.  Since $J \cap R = I$ (since $\pi(E) = W$), we have the induced diagram
\[
\xymatrix{
F^e_* R/I \ar@{^{(}->}[d]  \ar[r]^{\overline\phi_R} & R/I \ar@{^{(}->}[d] \\
F^e_* R'/J \ar[r]_{\overline\phi_{R'}} & R'/J.
}
\]
Taking fields of fractions gives us exactly the claimed diagram.
\end{proof}

We obtain the following corollary.

\begin{corollary}
\label{cor.CenterToPushfwdSeparable}
With notation and assumptions as in and above \autoref{prop.PInverseOnFields}, every element of $K(E)$ which is algebraic over $K(W)$ is separable over $K(W)$, in particular $\pi_* \O_E \supseteq \O_W$ is separable.  Furthermore, if $K(W) \subseteq L \subseteq K(E)$ and $K(W) \subseteq L$ is algebraic and separable, then $\phi_E$ restricts to a map $\phi_L : F^e_* L \to L$.
\end{corollary}
\begin{proof}
For the first part, we simply refer to \cite[Example 5.1]{SchwedeTuckerTestIdealFiniteMaps}.

The second part is essentially the argument given in \cite[Proposition 5.2]{SchwedeTuckerTestIdealFiniteMaps} in a slightly different context.  Note we use the ${1/p^e}$ notation instead of the $F^e_*$ notation in what follows.  We first consider $\phi_E|_{F^e_* L}$.  Since $F^e_* L = L^{1/p^e} = L \otimes K(W)^{1/p^ee} = L K(W)^{1/p^e} \subseteq K(E)^{1/p^e}$, for any arbitrary $l^{1/p^e} \in L^{1/p^e}$, write $l^{1/p^e}= \sum_i l_i w_i^{1/p^e}$ for some $l_i \in L \subseteq K(E)$ and $w_i \in K(W)$.  Then $\phi_E(l^{1/p^e})  = \phi_E(\sum_i l_i w_i^{1/p^e}) = \sum_i l_i \phi_E(w_i^{1/p^e})$.  But $\phi_E(w_i^{1/p^e}) = \phi_W(w_i^{1/p^e}) \in K(W)$ thus $\phi_E(l^{1/p^e}) \in L$.  Hence $\phi_L := \phi_E|_L : L^{1/p^e} \to L$ is an extension of $\phi_W$ and extends to $\phi_E$.
\end{proof}

Rephrasing \autoref{prop.PInverseOnFields} in the non-local case, we obtain the following result.

\begin{corollary}
\label{cor.PInverseExtendsOnSheaves}
Suppose that $(X, \Delta)$ is a pair in characteristic $p > 0$ with $(p^e - 1)(K_X + \Delta)$ Cartier.  Suppose further that $\pi : X' \to X$ is a proper birational map from a normal variety $X'$ and that $\pi^*(K_X + \Delta) = K_{X'} + \Delta_{X'}$.  Assume that $E \subseteq X'$ is a normal prime exceptional divisor of discrepancy $-1$ dominating an integral subscheme $W \subseteq X$ such that $(X, \Delta)$ is $F$-pure, and hence log canonical, at the generic point of $W$.  Note that this implies that $\Delta$ is effective at the generic point of $W$.  Suppose that $\pi|_E : E \to W$ can be written as $E \xrightarrow{\rho} W' \to W$ where $W' \to W$ is birational and $W'$ is normal.
Then there is a commutative diagram:
\[
\xymatrix{
F^e_* \O_E( (1-p^e)(K_E + \Delta_{E})) \ar[r]^-{\phi_E} & K(E) \\
 F^e_* \rho^{-1} \O_{W'}( (1-p^e)(K_{W'} + \Delta_{W', \Fdiff})) \ar@{^{(}->}[u]\ar[r]_-{\phi_{W'}} & \rho^{-1} K(W)\ar@{^{(}->}[u]
}
\]
Here $\Delta_E$ is the different of $K_{X'} + \Delta_{X'}$ along $E$ and $\Delta_{W', \Fdiff}$ is the $F$-different of $(X, \Delta)$ along $W'$ which maps to the the $F$-pure center $W$ (recall that the $F$-different is a {\bf b}-divisor by \autoref{defn.FDifferent} and \autoref{rem.FDiffBDivisor}).
\end{corollary}
\begin{proof}
This is just a non-local rephrasing of \autoref{prop.PInverseOnFields}.  Note
\[
\rho^* (p^e-1)(K_{W'} + \Delta_{W', \Fdiff}) \sim (p^e - 1)(K_E + \Delta_E)
\]
by construction.
\end{proof}

We immediately obtain the following.

\begin{corollary}
\label{cor.FDiffGeqDivDiff}
Suppose $(X, \Delta)$ is a pair such that $K_X + \Delta$ is $\bQ$-Cartier with index not divisible by $p$.  If $W \subseteq X$ is a log canonical center that is also an $F$-pure center of $(X, \Delta)$, then ${\bf \Delta}_{W, E, \Div} \leq {\bf \Delta}_{W, \Fdiff}$ as {\bf b}-divisors (here $E$ can be taken as any $E \to W$ such that $E$ has discrepancy $-1$).
In other words, for any proper and birational $W' \to W$ with $W'$ normal and $E \to W$ factoring through $W'$, we have that $\Delta_{W', \Div} \leq \Delta_{W', \Fdiff}$.  Furthermore, $\Delta_{W', \Div}$ is an honest $\bR$-divisor.
\end{corollary}
\begin{proof}
We fix $\pi_E : E \to W$ factoring through $W'$, where $E \subseteq X'$ is some normal prime divisor with discrepancy $-1$.  We apply \autoref{cor.PInverseExtendsOnSheaves} and obtain the displayed diagram therein. But this implies that the hypotheses of \autoref{lem.DivisorialPartOfDifferent} are satisfied.  We notice that $\Delta_{W', \Fdiff}$ is induced by the map labeled $\phi_{W'}$ in the same diagram.  We conclude that $\Delta_{W', \Div} \leq \Delta_{W', \Fdiff}$ and also that $\Delta_{W', \Div}$ is an honest $\bR$-divisor by \autoref{lem.DivisorialPartOfDifferent}.
\end{proof}

It would be natural to show that $\O_W \subseteq \pi_* \O_E$ is birational.  Using standard results from the theory of Frobenius splittings, we can do that under the following situation.\footnote{Which is frequently obtainable if we have access to the minimal model program.}
\begin{lemma}
With notation and assumptions as in \autoref{cor.PInverseExtendsOnSheaves}, assume additionally that the following holds:
\begin{itemize}
\item[(a)] There is a neighborhood $U \subseteq X$ of the generic point of $W$ such that $\Delta_{X'}|_{\pi^{-1}(U)} \geq 0$.
\item[(b)]  $\pi$ is an isomorphism over $X \setminus W$.
\end{itemize}
Then $\O_W \subseteq \pi_* \O_E$ is birational.
\end{lemma}
\begin{proof}
We localize the entire setup at the generic point of $W$ and hence work locally over the base.  Our hypotheses in (a) can then be translated into assuming that $X'$ is $e$-iterated Frobenius split compatibly with $E$.  This implies that $R^1 \pi_* \O_{X'}(-E) = 0$ by \cite[1.2.12 Theorem]{BrionKumarFrobeniusSplitting}, see \cite{MehtaVanDerKallenGRVanishing}.  Hence $\pi_* \O_{X'} \to \pi_* \O_E$ surjects.  But we know $\pi_* \O_{X'} = \O_X$ since $X$ is normal and so $\pi_* \O_E$ is a quotient of $\O_X$.  Now certainly $I_W$ annihilates $\O_E$ and so we have the factorization
\[
\O_X \to \O_W \hookrightarrow \pi_* \O_E.
\]
The composition is surjective and hence $\O_W \to \pi_* \O_E$ is an isomorphism.  Thus unlocalizing we conclude that $\O_W \subseteq \pi_* \O_E$ is birational as desired.
\end{proof}

\section{A canonical bundle formula and the local structure of the $F$-different}
\label{sec.CanBundleFormula}

The first few results of this section work under the following notation and hypotheses.

\begin{setting}
\label{set.CanBundFormula}
Assume that $\pi : E \to W$ is a proper dominant map between normal $F$-finite integral schemes of characteristic $p > 0$ and that $\pi_* \O_E = \O_W$. Suppose that $\Delta_E$ is a $\bQ$-divisor on $E$ and that $\O_E((1 - p^e)(K_E + \Delta_E)) \cong \pi^* \sL_W$ for some line bundle $\sL_W$, hence $(p^e - 1)(K_E + \Delta_E)$ is Cartier and linearly equivalent to a Cartier divisor pulled back from $W$.
\end{setting}

\begin{theorem}[Canonical bundle formula]
\label{thm.CanonicalBundleFormla}
With notation as in \autoref{set.CanBundFormula}, suppose that every horizontal component of $\Delta_E$ (those components which dominate $W$) is effective.  If the generic fiber $(E_{\eta}, \Delta_{E,\eta})$ is globally $F$-split relative to $\pi$, then there exists a canonically determined divisor $\Delta_W$ on $W$ such that $\pi^*(K_W + \Delta_W) \sim_{\bQ} K_E + \Delta_E$.
\end{theorem}
\begin{proof}
Suppose that $\phi_{E} : F^e_* \O_E( (1-p^e)(K_E + \Delta_E)) \to K(E)$ is a map corresponding to $\Delta_E$ (these maps are unique up to multiplication by units of $\Gamma(E, \O_E)$, which are the same as units of $\Gamma(W, \O_W)$).

We work locally on the base $W$, and hence assume that $W$ is affine.  Since the horizontal part of $\Delta_{E}$ is effective, we have a Cartier divisor $D \geq 0$ on $W$ such that $\Image(\phi_E) \subseteq \O_E(\pi^* D) \subseteq K(E)$, and so we push down $\phi_E$ by $\pi$ and by the projection formula obtain:
\[
\phi_W := \pi_* \phi_E : F^e_* \sL_W \to \O_W(D) \subseteq K(W).
\]
Since $(E_{\eta}, \Delta_{E, \eta})$ is globally $F$-split relative to $\pi$, we see that $\phi_W = \pi_* \phi_E$ is nonzero.
The divisor $\Delta_W$ associated to $\phi_W$ satisfies the desired condition since $\pi^* (1-p^e)(K_W + \Delta_W) \sim (1-p^e)(K_E + \Delta_E)$.  Note that $\Delta_W$ is independent of choices, since, as already noted, multiplication by units of $\Gamma(W, \O_W)$ is the same as multiplication by units of $\Gamma(E, \O_E)$.
\end{proof}

\begin{corollary}
With notation as in \autoref{thm.CanonicalBundleFormla}, we may always choose $\phi_W : F^e_* \O_W(K_W + \Delta_W) \to K(W)$ corresponding to $\Delta_W$ and $\phi_E : F^e_* \O_X( (1-p^e)(K_E + \Delta_E)) \to K(E)$ corresponding to $\Delta_E$ such that $\phi_W$ extends to $\phi_E$ as in \autoref{cor.PInverseExtendsOnSheaves}.
\end{corollary}
\begin{proof}
This follows from the construction of $\phi_W$ in the proof of \autoref{thm.CanonicalBundleFormla}.
\end{proof}

We can also work in the following more general setting without the effectivity hypothesis on $\Delta_E$

\begin{corollary}[General canonical bundle formula]
\label{cor.GenCanonicalBundleFormula}
With notations as in \autoref{set.CanBundFormula}, suppose that $\phi_{E} : F^e_*\O_E( (1-p^e)(K_E + \Delta_E)) \to K(E)$ is a map corresponding to $\Delta_E$ that satisfies the following conditions:
\begin{itemize}
\item[(i)]  $\phi_{E}\big(F^e_* \pi^{-1} \sL_W \big) \subseteq \pi^{-1}\big(K(W)\big) \subseteq K(E)$.
\item[(ii)]  The map induced via $\pi_*$, $\phi_W : F^e_* \sL_W \to K(W)$, is nonzero.\footnote{This means that the generic fiber is sub-$F$-split.}
\end{itemize}
Then there exists a canonically determined $\Delta_W$ (corresponding to $\phi_W$) such that $K_E + \Delta_E \sim_{\bQ} \pi^* (K_W + \Delta_W)$.
\end{corollary}
\begin{proof}
The proof is the same as the last few lines of the proof of \autoref{thm.CanonicalBundleFormla}.
\end{proof}
\begin{remark}
\label{rem.EasyCanBundleSingularities}
Note that in the setup of \autoref{cor.GenCanonicalBundleFormula}, if $(E, \Delta_E)$ is globally sub-$F$-split relative to $\pi$, then $(W, \Delta_W)$ is locally sub-$F$-pure, also compare with \cite{MehtaRamanathanFrobeniusSplittingAndCohomologyVanishing}.
\end{remark}

\begin{corollary}
\label{cor.FadjunctionViaCanBundle}
Suppose that $V$ is an $F$-pure center of $(X, \Delta)$ that is also a log canonical center with $E$ a normal divisor with discrepancy $-1$ (on some birational model $\pi : X' \to X$) mapping to $V$.  Write $K_{X'} + \Delta_{X'} = \pi^*(K_X + \Delta)$ and set $V'$ to be the normalization of $V$ and form the Stein factorization $E \to W \xrightarrow{h} V' \to V$ (noting that $W$ is normal since $E$ is).  Also suppose that $\Delta_E$ is the different of $K_{X'}+\Delta_{X'}$ along $E$ and that $\Delta_{V'}$ is the $F$-different of $(X, \Delta)$ along $V'$.  Then $\Delta_E$ induces $\Delta_W$ by \autoref{cor.GenCanonicalBundleFormula} and we have that $\Delta_W = h^* \Delta_{V'} - \Ram_{W/V'}$.  In particular, if $W = V'$, then $\Delta_E$ induces the $F$-different of $(X, \Delta)$ along $W$.
\end{corollary}
\begin{proof}
We notice that $h : W \to V'$ is separable by \autoref{cor.CenterToPushfwdSeparable} and furthermore, we have the restriction $\phi_E|_{F^e_* K(W)} : F^e_* K(W) \to K(W)$ which extends the map $\phi_{V'} : F^e_* \O_{V'}((1-p^e)(K_{V'} + \Delta_{V'}) \subseteq F^e_* K(V') \to K(V')$ from which we compute the $F$-different.  It follows that the conditions of \autoref{cor.GenCanonicalBundleFormula} hold and that the induced map $\phi_W$ generically agrees with $\phi_{E}|_{F^e_* K(W)}$.  Hence $\phi_{V'}$, which computes the $F$-different on $V'$, extends to $\phi_W$.  This extension is unique and corresponds to $h^* \Delta_{V'} - \Ram_{W/V'}$ by \cite{SchwedeTuckerTestIdealFiniteMaps}.
\end{proof}

We describe the structure of $\Delta_E$ via analogy with the usual different.  In that setting \autoref{dfn:div-part}, the numbers $c_Q$ are defined by \emph{local} sub-log canonicity.  Here we instead require \emph{global} sub-$F$-splitting.

\begin{proposition}
\label{prop.DescriptionOfFDifferentFromFormula}
With notation as in \autoref{set.CanBundFormula} and suppose that $\Delta_E$ induces $\Delta_W$ via \autoref{thm.CanonicalBundleFormla} or \autoref{cor.GenCanonicalBundleFormula}.

For each point $Q \in W$ of codimension 1, set
\[
d_Q = \sup \{ t \;|\; (E, \Delta_E + t \pi^* V(Q) ) \text{ is globally sub-$F$-split over a neighborhood of $Q$} \}.
\]
Then
\[
\Delta_W = \sum (1-d_Q) V(Q).
\]
\end{proposition}
\begin{proof}
Since $W$ is normal, codimension $1$ points of $W$ correspond to Weil divisors. For each codimension $1$ point $Q \in W$ we set
\[
b_Q = \sup \{ t \;|\; (W, \Delta_W + t V(Q) ) \text{ is locally sub-$F$-pure in a neighborhood of $Q$} \}.
\]
Note that $b_Q$ is defined using only the data $(W, \Delta_W)$.  Working locally near $Q$ on $W$, where the stalk $\O_{W,Q}$ is a DVR, we see that $(W, \lambda Q)$ is sub-$F$-pure in a neighborhood of $Q$ if and only if $\lambda \leq 1$.   Hence we immediately see that:
\[
\Delta_W = \sum (1-b_Q) V(Q).
\]
The proof will then be completed by the following claim.
\begin{claim}
\label{clm.LocFPureEqualsGlobFSplit}
Suppose $t$ is a rational number without $p$ in its denominator.  The following are equivalent:
\begin{enumerate}
\item{} $(W, \Delta_W + tV(Q))$ is locally sub-$F$-pure in a neighborhood of $Q$.
\label{clm.cond1}
\item{} $(E, \Delta_E + t\pi^*(V(Q)))$ is globally sub-$F$-split over a neighborhood of $Q$.
\label{clm.cond2}
\end{enumerate}
\end{claim}
\begin{proof}[Proof of claim]
We fix a $Q$ and work on an affine neighborhood of $Q$ where $V(Q)$ is a Cartier divisor generated by a single element $w \in \O_W$.  Choose $e$ so that $(p^e - 1)t \in \bZ$.
Because we assumed that $\Delta_W$ was induced by $\Delta_E$ via \autoref{thm.CanonicalBundleFormla}, condition \autoref{clm.cond1} says that
\[
1 \in \Image\Big( H^0\big(W, F^e_* (w^{t(p^e-1)} \cdot \pi_*\O_E( (1-p^e)(K_E + \Delta_E)))\big) \xrightarrow{\phi_W} H^0(W, K(W)) \Big).
\]
But by the projection formula this is the same condition that $1$ is in the image of the map:
\[
\begin{array}{rl}
& H^0(E, F^e_* \O_E( (1-p^e)(K_E + \Delta_E + t \pi^* V(Q)))) \\
\xrightarrow{\phi_E} & H^0(E, K(E)).
\end{array}
\]
This proves the claim.
\end{proof}
The claim proves the proposition.
\end{proof}

Our next goal is to study $\Delta_W$ in terms of the structure of the fibers of $E \to W$.  We work locally and now typically assume that $W$ is the spectrum of a DVR and hence we identify $\O_W$ with the line bundles $\sL_W = \O_W( (1-p^e)(K_W + \Delta_W))$ studied earlier.

\begin{lemma}
\label{lem.EqualCoefficientsImplyFSplitFiber}
Let $\pi : E \to W$ be a proper dominant map between $F$-finite integral schemes with $\pi_* \O_E = \O_W$
 such that $W$ is the spectrum of a DVR, $W = \Spec A$.  Suppose that we have a nonzero map $\phi_W : F^e_* \O_W \to K(W)$ that extends to a map $\phi_E : F^e_* \O_E \to K(E)$ via the inclusion $K(W) \subseteq K(E)$ and with corresponding $\Delta_W$ and $\Delta_E$.  Let $B = \sum b_i {B_i}= \pi^* Q$ be the fiber over the closed point $Q \in W$.

Suppose that $c = \coeff_{B_i}(\pi^* \Delta_{W}) 
= \coeff_{B_i}(\Delta_{E})$ for some $B_i$.  Then $b_i = 1$.  Furthermore, in this case, let $\phi_{B_i} : F^e_*\O_{B_i} \to K(B_i)$ be the map on $B_i$ induced by $\Delta_E + (1-c)B$ as in \autoref{eq.AdjunctionInducingFDiff}.  Then $(B_i, \Delta_{\phi_{B_i}})$ is globally sub-$F$-split.
\end{lemma}
\begin{proof}
Let $w \in A$ denote the local parameter for $Q \in W$.
Since $\coeff_Q(\Delta_{W}) = c/b_i$ is a rational number without $p$ in the denominator, we observe that $(W, \Delta_W + (1-c/b_i) Q) = (W, Q)$ is $F$-pure and so $(E, \Delta_E + (1-c/b_i) B)$ is globally sub-$F$-split by \autoref{clm.LocFPureEqualsGlobFSplit}.  We can choose $\psi_W$ and $\psi_E$ the corresponding maps.  We then observe that the coefficient $\coeff_{B_i}(\Delta_E + (1-c/b_i)B) = c + b_i - c = b_i \geq 1$.  Hence $b_i B_i$ is compatible with $\psi_E$ (meaning at the generic point of $B_i$ that $\psi_E(F^e_* (\sL(-b_i B_i)_{\eta_{B_i}})) \subseteq \O_E(-b_i B_i)_{\eta_{B_i}}$).  This induces a map on $b_i B_i$ as in \autoref{subsec.Fdifferent} and by construction, its image contains $1$ (both locally and globally).  Working on an affine chart of $B_i$, we see that $b_i B_i$ must be reduced and so $b_i = 1$.  The rest of the result follows.
\end{proof}

In the case that the fiber $B$ is integral and $\Delta_E \geq 0$, the above lemma simply states that if the coefficient $c = \coeff_B \Delta_E$ is equal to the coefficient of $B$ in $\Delta_W$, then $(B, \Delta_B)$ is $F$-split (where $\Delta_B$ is the different of $\Delta_E + (1-c) B$ along $B$).

\begin{lemma}
\label{lem.FSplitFiberImpliesEqualCoeffs}
Let $\pi : E \to W$ be a proper dominant map between $F$-finite normal schemes with geometrically connected and reduced fibers, where $W$ is the spectrum of a DVR, $W = \Spec A$, and $\O_W = \pi_* \O_E$.  Suppose now that $\Delta_E$ is a $\bQ$-divisor on $E$ satisfying the conditions of \autoref{cor.GenCanonicalBundleFormula} and let $\Delta_W$ be the divisor on $W$ induced from $\Delta_E$.  Write $\Delta_E = \Delta_E^h + \Delta_E^v$ a decomposition into the horizontal and vertical parts . 

Let $\pi^* Q = B \subseteq E$ denote the closed fiber and suppose that $B$ is integral.  Let $\phi_B : F^e_* \sL_{\Delta_E^h+B}|_B \to K(B)$ denote the map induced by $\phi_{\Delta_E^h + B} : F^e_* \sL_E := F^e_* \O_E( (1-p^e)(K_E + \Delta_E^h + B)) \to K(E)$ via the method of \autoref{eq.AdjunctionInducingFDiff}.

If $\phi_B$ has $1$ in its image on global sections\footnote{This just means we can identify $\phi_B$ with a sub-$F$-splitting since abstractly $\sL_E \cong \O_E$.}, then $\coeff_Q(\Delta_W) = \coeff_B(\Delta_E)$.
\end{lemma}
\begin{proof}
Let $c = \coeff_Q(\Delta_W)$ and let $w \in A$ be the local parameter.  Consider the map $\phi_{\Delta_E^h} : F^e_* \O_E ( (1-p^e)(K_E + \Delta_E^h)) \to K(E)$ and the induced map $\psi_W : F^e_* \sL_W \to K(W)$ on $W$ obtained by applying $\pi_*$ (note that the image really does lie in $K(W)$ since the image of $\pi_* \phi_{\Delta_E}$ lies in $K(W)$).  Let $\Delta_W^h$ denote the corresponding divisor.

\begin{claim}
\label{clm.Effectivity-EqualCoefficientsImplyFSplitFiber}
With notation as above, $\Delta^h_W \geq 0$.
\end{claim}
\begin{proof}[Proof of claim]
It suffices to show that the image of $\psi_W$ is contained in $\O_W$.  Certainly $\Image(\psi_W) = w^n \O_W$ for some $n \in \bZ$ where $w$ is the local parameter of the DVR $A$.  If $n < 0$, this implies that $w^{-1} \in \Image(\psi_W) = \Image(\pi_* \phi_{\Delta_E^h})$.  But $w^{-1}$ generates $\O_E(B) \subseteq K(E)$ as an $\O_E$-module.  This means that $\O_E(B) \subseteq \Image(\phi_{\Delta^h_E})$ which forces $B$ to be a component of the negative part of $\Delta^h_E$, see our construction in \autoref{subsec.MapsAndQDivisors}.  The latter is impossible since $\Delta^h_E$ is horizontal.  This proves the claim.
\end{proof}
We return to the proof of \autoref{lem.FSplitFiberImpliesEqualCoeffs}.
For any integer $a$, we notice that the induced map
\[
\psi_{W,a} : F^e_* (w^a \cdot \sL_W) \to K(W),
\]
induced from $\psi_W$, corresponds to the divisor $\Delta_W^h + {a \over p^e - 1} Q$, and is itself induced from
\[
\xymatrix{
\phi_{\Delta_E^h,a} : F^e_* (w^a \cdot \O_E ( (1-p^e)(K_E + \Delta_E^h))) \ar@{=}[d] \ar[r] & K(E)\\
\phi_{\Delta_E^h + {a \over p^e - 1} B} : F^e_* (\O_E ( (1-p^e)(K_E + \Delta_E^h + {a \over p^e - 1} B)))\ar[ur]
}
\]
via pushforward, a map which corresponds to $\Delta_{E}^h + {a \over p^e - 1} B$.  Note that $\Image(\phi_{\Delta_E^h+B}) \subseteq \O_E(G) \subseteq K(E)$ where $G$ is some effective horizontal Weil divisor (supported where $\Delta^h_E$ is not effective).

Notice also that the image of the global sections under $\phi_{\Delta_E^h+B}$ lies in $\pi^{-1} K(W)$.

Consider the following diagram induced as in \autoref{subsec.Fdifferent}:
\begin{equation}
\label{eq.FSplitFiberImpliesEqualCoeffs}
\xymatrix{
\Gamma(E, F^e_* \sL_E) \ar[d] \ar[r]^-{\phi_{\Delta_E^h + B}} & \Gamma(E, \O_E(G)) \ar[d]^{\delta} \\
\Gamma(B, F^e_* \sL_E|_B) \ar[r]_-{\phi_B} & \Gamma(B, K(B))
}
\end{equation}
The left vertical map is clearly surjective (since abstractly this is just $F^e_* A = \Gamma(E, F^e_* \O_E) \to \Gamma(B, F^e_* \O_B) = F^e_* A/wA$) and the bottom horizontal map has a global section mapping to $1$ by hypothesis.  Hence there exists a $\beta = \phi_{\Delta_E^h + B}(\alpha) \in \Gamma(E, \O_E(G))$ with $\delta(\beta) = 1  \in \Gamma(B, K(B))$.

\begin{claim}
The image $\phi_{\Delta_E^h + B}(\Gamma(E, F^e_* \sL_E))$ in $\Gamma(E, \O_E(G))$ equals $A$ and in particular contains $1$.
\end{claim}
\begin{proof}[Proof of claim]
We already noted that the image is contained in $K(A) = \Gamma(W, K(W))$.  It is also clearly an $A$-module.  Since $A$ is a DVR, we see that the image is equal to $w^n A$ for some $n \in \bZ$.  It suffices to show that $n = 0$. 
First notice that the image of $\phi_{\Delta^h_E+B}$ is contained in the image of $\psi_W$ which we already showed was contained inside $A$ in \autoref{clm.Effectivity-EqualCoefficientsImplyFSplitFiber}.  Thus consider $\beta \in \Gamma(E, \O_E(G))$ (defined above with $\delta(\beta) = 1$) as an element of $A$ but also as an element of $\O_{E, \eta_B} \supseteq \Gamma(E, \O_E(G))$ where $\eta_B$ is the generic point of $B$ (we are just using the fact that global section become sections at stalks).   The map $\delta$ of \autoref{eq.FSplitFiberImpliesEqualCoeffs}, at the stalk of $\eta_B$, is induced by
\[
\O_{E, \eta_B} \to \O_{E, \eta_B} / (w \cdot \O_{E, \eta_B}) = K(B).
\]
Since $\delta(\beta) = 1$, it follows that $\beta \notin w \cdot \O_{E, \eta_B} \supseteq w \cdot A$.  Thus $\beta \in A$ is a unit.  The claim follows.

\end{proof}
We return to the proof of \autoref{lem.FSplitFiberImpliesEqualCoeffs}.  The fact that the top horizontal map of \autoref{eq.FSplitFiberImpliesEqualCoeffs} has $1$ in its image forces $(W, \Delta_W^h + Q)$ to be $F$-pure and hence that $\coeff_Q(\Delta_W^h + Q) \leq 1$.  Since $\Delta_W^h \geq 0$ by \autoref{clm.Effectivity-EqualCoefficientsImplyFSplitFiber} this means that $\coeff_Q(\Delta_W^h) = 0$ and so $\Delta_W^h = 0$.  But now we have just proven that $\Delta_E^h + {a \over p^e - 1} B$ induces ${a \over p^e - 1}Q$ on $W$.  In particular, the vertical part of the fiber has the same coefficient as $\Delta_W$ as claimed.
\end{proof}

We rephrase this in a special case.  Note that since $B$ is Cartier, if we write $\Delta_E^h|_B = \Delta_B$, then $\Delta_B$ is the $F$-different of $(E, \Delta_E^h)$ \cite{DasStronglyFRegularInversion}, \cite[Proposition 7.2]{SchwedeFAdjunction}.

\begin{corollary}
\label{cor.FSplitFiberImpliesEqualCoefficients}
Let $\pi : E \to W$ be a proper dominant map between $F$-finite normal schemes with geometrically connected and normal fibers, where $W$ is the spectrum of a DVR, $W = \Spec A$, and $\O_W = \pi_* \O_E$.  Let $Q$ be the closed point of $W$ and write $B = \pi^* Q$.  Suppose that $\Delta_E = \Delta_E^h + cB$ is a $\bQ$-divisor on $E$ satisfying the conditions of \autoref{cor.GenCanonicalBundleFormula}.  If $\Delta_E^h|_B = \Delta_B$ is the different on $B$, and $(B, \Delta_B)$ is globally sub-$F$-split, then $c = \coeff_B(\Delta_E) = \coeff_Q(\Delta_W)$ where $\Delta_W$ is induced by \autoref{thm.CanonicalBundleFormla}.
\end{corollary}

Putting together \autoref{cor.FSplitFiberImpliesEqualCoefficients} and \autoref{lem.EqualCoefficientsImplyFSplitFiber} we obtain the following.

\begin{corollary}
\label{cor.IsFrobSplitting}
Suppose $\pi : E \to W$ is a flat proper dominant map between $F$-finite normal integral schemes with geometrically connected and normal fibers and $\Delta_E$ is a $\bQ$-divisor on $E$ satisfying the conditions of \autoref{cor.GenCanonicalBundleFormula}.   Let $\Delta_W$ be the divisor on $W$ induced by $\Delta_E$.  Then the vertical components of $\Delta_E - \pi^* \Delta_{\phi_W}$ correspond to the fibers $E_t = \pi^{-1} t$ over codimension $1$ points $t \in W$, such that $(E_t, \Delta_{E_t})$ is not sub-$F$-split where $\Delta_{E}|_{E_{t}} = \Delta_{E_t}$ is the different.
\end{corollary}
\begin{proof}
We work at the stalk of $t \in W$ which is a DVR.  Then if the fiber $(E_t, \Delta_{E_t})$ is sub-$F$-split, we see that $\coeff_{E_t}(\Delta_E) = \coeff_t(\Delta_W)$ by \autoref{cor.FSplitFiberImpliesEqualCoefficients}.  Conversely, if $\coeff_{E_t}(\Delta_E) = \coeff_t(\Delta_W)$ then $(E_t, \Delta_{E_t})$ is sub-$F$-split by \autoref{lem.EqualCoefficientsImplyFSplitFiber}.
\end{proof}

This corollary completely matches our example before with the family of elliptic curves.  In fact, especially if one starts in characteristic zero, one can frequently base change the family $E\to W$ (log canonical place mapping to a log canonical center), followed by a birational transformation of $E$, so that the family obtained is exactly of the above type.  Indeed, compare with \cite[Lemma 4.7, Theorem 4.8]{DasHacon}.  We briefly discuss how to perform this procedure and explain how the results we have so far explained show how to keep track of the $F$-different as each step is being performed.

\begin{algorithm}
\label{alg.GeometricFDifferentAlg}
Suppose that $\pi : E \to W$ is a proper dominant map between normal integral $F$-finite schemes with $W$ $1$-dimensional and $\Delta_E$ is a divisor on $E$ such that if $\Delta_E = \Delta_E^h + \Delta_E^v$ is the decomposition into the vertical and horizontal parts, then $\Delta_E^h \geq 0$.  Suppose further that $L_E := (1 - p^e)(K_E + \Delta_E) \sim \pi^* L_W$.

Finally suppose that $\Delta_E$ corresponds to $\phi_E : F^e_* \O_E(L_E) \to K(E)$ and restricts to a nonzero map $\phi_W : F^e_* \O_W(L_W) \to K(W)$ as in \autoref{cor.PInverseExtendsOnSheaves}.

\begin{itemize}
\item[Step 1]  Let $\pi : E \xrightarrow{\pi_1} W_1 \xrightarrow{\rho} W$ be the Stein factorization.  By \autoref{cor.CenterToPushfwdSeparable}, $\rho : W_1 \to W$ is separable and in fact the map $\phi_E$ restricts to a map $\phi_{W_1} : F^e_* \O_{W_1}(\rho^* L_W) \to K(W_1)$ which extends $\phi_W$.  Of course, $\pi_1$ has geometrically connected fibers.  The generic fiber $(E_{\eta}, \Delta_{\eta})$ is generically $F$-split since the map $\phi_E$ extends a nonzero map.  Hence since $(\pi_1)_* \O_{E} = \O_{W_1}$, the $E_{\eta}$ is geometrically reduced and hence so are the closed fibers of $\pi_1$ outside a proper closed subset of lower codimension.
    \begin{center}
    We notice that $\phi_{W_1}$ is induced from $\phi_E$ by pushforward by $(\pi_1)_*$.
    \end{center}
\item[Step 1a]  While we hope that $W_1 = W$, we do not know whether it holds in general.  However, because $\rho$ is separable, we have the divisor $\Delta_{W_1}$ corresponding to $\phi_{W_1}$ given by the formula $\Delta_{W_1} = \rho^* \Delta_{W} - \Ram_{W_1/W}$, see \autoref{subsec.pMapsAndBaseExtension}.
    \begin{center}
    Thus to compute $\Delta_W$, it is sufficient to compute $\Delta_{W_1}$.
    \end{center}
\item[Step 2]  Suppose that there exists a proper dominant and separable map $\rho_2 : W_2 \to W_1$ such that the multi-sections of $E \to W_1$ corresponding to the components of $\Delta^h_E$ become honest sections of $E_2 \to W_2$ where $E_2$ is the normalization of the component of $E \times_{W_1} W_2$ that dominates $W_2$.  Let $\kappa : E_2 \to E,\pi_2 : E_2 \to W_2$ be the induced maps.  We notice by \autoref{prop.keyBaseExtensionsPMapsForFamilies}, our maps $\phi_{E}$ and $\phi_{W_1}$ extend uniquely to maps $\phi_{E_2}$ on $E_2$ and $\phi_{W_2}$ on $W_2$ and the corresponding divisors are $\Delta_{E_2} = \kappa^* \Delta_{E} - \pi_2^* \Ram_{W_2/W_1}$ and $\Delta_{W_2} = \rho_2^* \Delta_{W_1} - \Ram_{W_2/W_1}$.  Furthermore, $\Delta_{E_2}$ induces $\Delta_{W_2}$ via \autoref{thm.CanonicalBundleFormla}.
    \begin{center}
    Thus to compute $\Delta_{W_1}$, it is sufficient to compute $\Delta_{W_2}$.
    \end{center}
\item[Step 3]
    Now suppose that there exists $E_3$ birational to $E_2$ with $E_2 \xleftarrow{\alpha} E' \xrightarrow{\beta} E_3$ birational morphisms with $\alpha^*(K_{E_2} + \Delta_{E_2}) = \beta^*(K_{E_3} + \Delta_{E_3})$.  Suppose further that $\Delta_{E_3}$ has effective horizontal part.  Finally assume that the fibers of $\pi_3 : E_3 \to W_2$ are geometrically normal and integral (see \cite{DasHacon} as mentioned above for cases when this can be arranged).

    Note that $\phi_{E_2}$ and $\phi_{E_3}$ can be chosen to be the same map generically, and hence they induce the same map $\phi_{W_2}$.  In particular:
    \begin{center}
    We can at least compute the locus where $\Delta_{W_2}$ differs from the vertical components of $\Delta_{E_3}$ via \autoref{cor.IsFrobSplitting}.
    \end{center}
\end{itemize}
\end{algorithm}

\section{The $F$-different under reduction to characteristic $p > 0$}

\begin{lemma}
\label{lem.CantMakeDivisorBiggerAtFPT}
Suppose that $(R, \bm)$ is a normal local $F$-finite ring of characteristic $p > 0$ and that $(X = \Spec R, \Delta \geq 0)$ is a pair in characteristic $p > 0$ such that $(p^e - 1)(K_X + \Delta)$ is integral and Cartier.  Suppose that $c = \fpt(X, \Delta, \bm)$.  Further suppose that $(X, \Delta, \bm^c)$ is sharply $F$-pure (so that $p$ does not divide the denominator of $c$, \cf \cite{SchwedeSharpTestElements}).  If $\Delta' > \Delta$ is another divisor with $\Delta' - \Delta$ $\bQ$-Cartier, then $(X, \Delta', \bm^c)$ is not sharply $F$-pure.
\end{lemma}
\begin{proof}
By replacing $e$ with a larger number, we may also assume that $(p^e - 1)c \in \bZ$.  Write $\Delta'-\Delta = t\Div(f)$.
Consider
\[
\begin{array}{rl}
& F^e_* (\bm^{c(p^e - 1)}) \cdot \Hom_R(F^e_* R( \lceil (p^e - 1)\Delta' \rceil ), R)\\
= & F^e_* (\bm^{c(p^e - 1)} \cdot f^{\lceil t(p^e - 1)\rceil}) \cdot \Hom_R(F^e_* R( (p^e - 1)\Delta ), R)\\
\subseteq & F^e_* (\bm^{c(p^e - 1)+1} ) \cdot \Hom_R(F^e_* R( (p^e - 1)\Delta ), R)\\
\subseteq & F^e_* (\bm^{c(p^e - 1)}) \cdot \Hom_R(F^e_* R( (p^e - 1)\Delta), R)\\
\xrightarrow{\text{eval@1}} & R.
\end{array}
\]
By definition, since $c$ is the $F$-pure threshold, we see that
\[
\Image\Big(F^e_* (\bm^{c(p^e - 1) + 1}) \cdot \Hom_R(F^e_* R( (p^e - 1)\Delta), R) \xrightarrow{\text{eval@1}} R \Big) \subseteq \bm
\]
and so indeed we see that $(X, \Delta', \bm^c)$ is not sharply $F$-pure.
\end{proof}

We refer to \cite{HochsterHunekeTightClosureInEqualCharactersticZero} for details of the reduction to characteristic $p > 0$ process.

\begin{theorem}
\label{thm.ModuliPartOfFdifferentMoves}
Let $(X, \Delta \geq 0)$ be a normal pair in characteristic zero with $K_X + \Delta$ $\bQ$-Cartier, and $W$ a normal LC-center of $(X, \Delta)$. Assume that the $\mathbf{b}$-divisor $\mathbf{K+\Delta}_{\div}$ descends to $W$ as a $\bQ$-Cartier divisor; in particular $K_W+\Delta_{W, \div}$ is $\bQ$-Cartier.
We consider the behavior of $(X_p, \Delta_p)$ after reduction to characteristic $p \gg 0$.

Assume the weak ordinarity conjecture \cite{MustataSrinivasOrdinary}.
Let $Q \in W$ be a point which is not the generic point of $W$.  Then there exists infinitely many primes $p > 0$ such that if $\Delta_{W_p, \Fdiff} \geq (\Delta_{W,\Div})_p$ is the $F$-different of $(X_p, \Delta_p)$ along $W_p$, then $\Delta_{W_p, \Fdiff} = (\Delta_{W,\Div})_p$ in a neighborhood of $Q_p$.  In other words, $Q_p$ is not contained in $\Supp(\Delta_{W_p, \Fdiff} - (\Delta_{W,\Div})_p)$.

\end{theorem}

\begin{proof}
Let $I_Q \subseteq \O_X$ be the ideal of $Q$ in $X$.  Let $c_Q$ denote the log canonical threshold of $(W, \Delta_{W, \Div}, (I_Q \cdot \O_W)^t)$.  By shrinking $X$, we may assume that this is the log canonical threshold at the generic point of $Q$.  By inversion of adjunction \cite{HaconLCInversionAdjunction}, we see that $(X, \Delta, I_Q^t)$ is log canonical at $Q$ if and only if $(W, \Delta_{W, \Div}, (I_Q \cdot \O_{W})^t)$ is log canonical.
We also fix a log canonical place $E$ over $Q$ in some model $X' \to X$.

We reduce this entire setup to characteristic $p > 0$.   By the weak ordinarity conjecture and its corollaries we see that
 $(X_p, \Delta_p, (I_{Q_p})_p^{c_Q})$ is sharply $F$-pure and likewise that
$(W_p, (\Delta_{W,\Div})_p, (I_{Q_p} \cdot \O_X)_p^{c_{Q}})$ is sharply $F$-pure for infinitely many $p \gg 0$ \cite[Theorem 2.11]{TakagiAdjointIdealsAndACorrespondence}.  In particular, $c_Q$ is then the $F$-pure threshold of those pairs as well.

Of course, the $F$-different satisfies $\Delta_{W_p, \Fdiff} \geq \Delta_{W_p, E_p, \Div} \geq (\Delta_{W, E, \Div})_p$ where the first inequality follows by \autoref{cor.FDiffGeqDivDiff} and the second follows from the fact that if something reduced from characteristic zero is log canonical in some $p \gg 0$, then it is log canonical in characteristic zero.\footnote{In fact, we can easily assume that $\Delta_{W_p, E_p, \Div}$ agrees with $(\Delta_{W, E, \Div})_p$ wherever the latter is supported and over such $Q_p$ that there is a vertical divisor on the different on $E$ mapping to $Q$, but we do not need this.}  Set $B = \Delta_{W_p, \Fdiff} - (\Delta_{W, \Div})_p$ and suppose that $B > 0$ near $Q$ in order to obtain a contradiction.  Note that $B$ is $\bQ$-Cartier since $K_{W_p} + \Delta_{W_p, \Fdiff}$  and $K_{W_p}+(\Delta_{W, \div})_p$ are both $\bQ$-Cartier.  Simply apply \autoref{lem.CantMakeDivisorBiggerAtFPT} to conclude that $(W_p, \Delta_{W_p, \Fdiff}, (I_Q \cdot \O_{W})^{c_Q})$ is not sharply $F$-pure.  On the other hand $(W_p, \Delta_{W_p, \Fdiff}, (I_Q \cdot \O_{W})^{c_Q})$ is sharply $F$-pure
by $F$-adjunction (see \cite{SchwedeFAdjunction}), since $(X_p, \Delta_p, I_Q^{c_Q})$ is sharply $F$-pure.
This yields a contradiction and the result follows.
\end{proof}


The main content of the above result is that, assuming the weak ordinarity conjecture, the moduli part of the different on $W'$ is somehow semi-ample (or at least has no fixed components) if we take the reduction to characteristic $p > 0$ procedure as a replacement for taking a general member of a linear system.  This holds as long as the ${\bf K} + {\bf \Delta}_{\Div}$ descends to $W$ as a $\bQ$-Cartier divisor.

\begin{corollary}
With notation and assumptions as in \autoref{thm.ModuliPartOfFdifferentMoves}, in particular still assuming the weak ordinarity conjecture, we know that for every prime divisor $D$ on $W$, there exists infinitely many $p > 0$ such that $\coeff_{D_p}\Delta_{W, \Fdiff} = \coeff_{D_p} (\Delta_{W, \Div})_p$.
\end{corollary}

\section{Further questions}

It would be nice to remove the descent hypothesis from \autoref{thm.ModuliPartOfFdifferentMoves}.

\begin{question}
It would be natural to generalize \autoref{thm.ModuliPartOfFdifferentMoves} in the following way.  Instead of simply fixing $W$, choose $W' \to W$ over which ${\bf K} + {\bf \Delta}_{\div}$ descends.  Then for each $Q \in W'$, we would like to show that the moduli part of the $F$-different, $\Delta_{W'_p, \Fdiff} - (\Delta_{W', \Div})_p$, does not pass through $Q$ (recall that the $F$-different and the divisorial part of the different are both {\bf b}-divisors).  It seems to the authors that one main problem is that if $c$ is the (sub)$\lct$ of $(W', \Delta_{W', \Div}, I_Q)$, then we do not know that $c$ is also the (sub)$\fpt$ of $(W'_p, (\Delta_{W', \Div})_p, I_Q)$ after reduction to characteristic $p > 0$ since $\Delta_{W', \Div}$ is not necessarily effective (even assuming the weak ordinarity conjecture).  See \cite{TakagiAdjointIdealsAndACorrespondence} and notice that S.~Takagi used a different definition of $F$-purity for pairs with non-effective divisors.

\end{question}

In this paper, we studied $F$-pure centers that are also log canonical centers.  However, there can be $F$-pure centers that are not log canonical centers, consider the origin in \cite[Example 7.12]{SchwedeTuckerTestIdealFiniteMaps}.

\begin{question}
Suppose that $W \subseteq X$ is an $F$-pure center of a pair $(X, \Delta)$ but suppose that $W$ is not a log canonical center.  Is there a geometric interpretation of the $F$-different $\Delta_W$?  Note that frequently such centers occur when their is some wild ramification over a divisor $E$ which dominates $W$.  Perhaps in that case there is some way to use that to characterize $\Delta_W$ by considering an alteration instead of a birational map?
\end{question}

\bibliographystyle{skalpha}
\bibliography{MainBib}

\end{document}